\documentclass[12pt,reqno]{amsart}

\usepackage{wasysym, amsmath, amssymb,graphicx,amsthm,latexsym, amsfonts, enumitem, mathtools, tensor}
\usepackage{hyperref}
\usepackage{xcolor}
\usepackage[all, color]{xy}
\usepackage{color}
\usepackage{float}
\usepackage{tikz}
\usetikzlibrary{arrows,decorations.pathmorphing,decorations.pathreplacing,positioning,shapes.geometric,shapes.misc,decorations.markings,decorations.fractals,calc,patterns}

\DeclareMathOperator{\Hom}{Hom}
\DeclareMathOperator{\End}{End}

\DeclareMathOperator{\mmod}{mod}

\DeclareMathOperator{\rad}{rad}
\DeclareMathOperator{\Ext}{Ext}

\DeclareMathOperator{\add}{add}

\theoremstyle{plain}
\newtheorem{theorem}{Theorem}[section]
\newtheorem*{theorem*}{Theorem}

\theoremstyle{definition}
\newtheorem{defn}[theorem]{Definition}

\newtheorem{exmp}[theorem]{Example} 
\newtheorem{remark}[theorem]{Remark}

\newtheorem{lemma}[theorem]{Lemma}

\newtheorem{setup}[theorem]{Setup}

\date{}
\begin{document}
\setlength{\parindent}{0pt}
\setlength{\parskip}{7pt}
\title[Auslander-Reiten $(d+2)$-angles in subcategories]{Auslander-Reiten $(d+2)$-angles in subcategories and a $(d+2)$-angulated generalisation of a theorem by Br\"uning}
\author{Francesca Fedele}
\address{School of Mathematics, Statistics and Physics,
Newcastle University, Newcastle upon Tyne NE1 7RU, United Kingdom}
\email{F.Fedele2@newcastle.ac.uk}

\maketitle
\begin{abstract}
Let $\Phi$ be a finite dimensional algebra over an algebraically closed field $k$ and assume gldim$\,\Phi\leq d$,  for some fixed positive integer $d$.
For $d=1$, Br\"uning proved that there is a bijection between the wide subcategories of the abelian category mod$\,\Phi$ and those of the triangulated category $\mathcal{D}^b(\text{mod }\Phi)$. Moreover, for a suitable triangulated category $\mathcal{M}$, J\o rgensen gave a description of Auslander-Reiten triangles in the extension closed subcategories of $\mathcal{M}$.

In this paper, we generalise these results for $d$-abelian and $(d+2)$-angulated categories, where kernels and cokernels are replaced by complexes of $d+1$ objects and triangles are replaced by complexes of $d+2$ objects.
The categories are obtained as follows: if $\mathcal{F}\subseteq \text{mod } \Phi$ is a $d$-cluster tilting subcategory, consider $\overline{\mathcal{F}}:=\text{add} \{ \Sigma^{id}\mathcal{F}\mid i\in\mathbb{Z}  \}\subseteq \mathcal{D}^b(\text{mod }\Phi)$. Then $\mathcal{F}$ is $d$-abelian and plays the role of a higher mod$\,\Phi$ having for higher derived category the $(d+2)$-angulated category $\overline{\mathcal{F}}$. 
\end{abstract}

\section{Introduction}
Let $d$ be a fixed positive integer, $k$ an algebraically closed field and $\Phi$ a finite dimensional $k$-algebra with global dimension at most $d$. The category of finitely generated (right) $\Phi$-modules is denoted by $\mmod \Phi$ and its bounded derived category by $\mathcal{D}^b(\mmod\Phi)$, with suspension functor $\Sigma$. Moreover, for an additive subcategory $\mathcal{C}$ of $\mmod\Phi$, we define an additive subcategory
\begin{align*}
\overline{\mathcal{C}}:=\add\{ \Sigma^{id} \mathcal{C}\mid i\in \mathbb{Z} \}\subseteq\mathcal{D}^b(\mmod\Phi).
\end{align*}
For $d\geq 2$, suppose there is a $d$-cluster tilting subcategory $\mathcal{F}\subseteq \mmod\Phi$. Then $\mathcal{F}$ plays the role of a higher $\mmod\Phi$ and $\overline{\mathcal{F}}$ of a higher derived category of $\mathcal{F}$.

We generalise Br\"uning's result on wide subcategories of $\mathcal{D}^b(\mmod\Phi)$ and J\o rgensen's result on Auslander-Reiten triangles in extension closed subcategories of triangulated categories to higher homological algebra.

\subsection{Classic background ($d=1$ case).}
In the case $d=1$, we have that $\mmod\Phi$ is hereditary. So \cite[theorem 1.1]{B} can be stated as follows in this case.

{\bf Theorem (Br\"uning).}
{\em
There is a bijection
\begin{align*}
    \Bigg\{
    \begin{matrix}\text{wide subcategories} \\ \text{of } \mmod\Phi \end{matrix}
    \Bigg\}
    \rightarrow
    \Bigg\{
    \begin{matrix} \text{wide subcategories} \\ \text{of } \mathcal{D}^b(\mmod\Phi) \end{matrix}
    \Bigg\}
\end{align*}
sending a wide subcategory $\mathcal{W}$ of $\mmod\Phi$ to $\overline{\mathcal{W}}$.
}

Happel introduced Auslander-Reiten triangles in triangulated categories in \cite{DH} and J\o rgensen studied Auslander-Reiten triangles in their extension closed subcategories in \cite{JP}.
Whenever $\mathcal{M}$ is a skeletally small $\Hom$-finite $k$-linear triangulated category with split idempotents and $\mathcal{W}\subseteq\mathcal{M}$ is an additive subcategory closed under extensions, \cite[theorem 3.1]{JP} states the following.

{\bf Theorem (J\o rgensen).}
{\em
Let $W$ be in $\mathcal{W}$ and suppose that there exists $U'$ in $\mathcal{W}$ and a non-zero morphism $W\rightarrow\Sigma U'$. Let
\begin{align*}
    \xymatrix{
    X\ar[r]& Y\ar[r]& W\ar[r]& \Sigma X
    }
\end{align*}
be an Auslander-Reiten triangle in $\mathcal{M}$. Then the following are equivalent.
\begin{enumerate}[label=(\alph*)]
    \item $X$ has a $\mathcal{W}$-cover of the form $U\rightarrow X$,
    \item there is an Auslander-Reiten triangle in $\mathcal{W}$ of the form
    \begin{align*}
    \xymatrix{
    U\ar[r]& V\ar[r]& W\ar[r]& \Sigma U.
    }
\end{align*}
\end{enumerate}
}

Note that the above theorem can be applied to any wide subcategory of the triangulated category $\mathcal{D}^b(\mmod\Phi)$. So, given a wide subcategory of $\mmod \Phi$, one can find a wide subcategory $\mathcal{W}$ of $\mathcal{D}^b(\mmod\Phi)$ using the theorem by Br\"uning and then use the theorem by J\o rgensen to find Auslander-Reiten triangles in $\mathcal{W}$.

\subsection{This paper ($d\geq 1$ case).}
%Throughout this paper, we work in the following setup.
%$Let $d$ be a fixed positive integer, $k$ an algebraically closed field and $\Phi$ a finite dimensional $k$-algebra with global dimension at most $d$. The category of finitely generated (right) $\Phi$-modules is denoted by $\mmod \Phi$ and its bounded derived category by $\mathcal{D}^b(\mmod\Phi)$, with suspension functor $\Sigma$. Moreover, for a subcategory $\mathcal{C}$ of $\mmod\Phi$, we define
%\begin{align*}
%\overline{\mathcal{C}}:=\add\{ \Sigma^{id} \mathcal{C}\mid i\in \mathbb{Z} \}\subseteq\mathcal{D}^b(\mmod\Phi).
%\end{align*}
%Concepts of homological algebra can sometimes be generalised to the so called higher homological algebra.
Jasso generalised abelian categories to $d$-abelian categories in \cite{JG}, where kernels and cokernels are replaced by complexes of $d+1$ objects, called $d$-kernels and $d$-cokernels respectively, and short exact sequences by complexes of $d+2$ objects, called $d$-exact sequences.
In \cite{GKO}, Geiss, Keller and Oppermann likewise generalised triangulated categories to $(d+2)$-angulated categories, where triangles are replaced by complexes consisting of $d+2$ objects.
%The same paper gives a recipe to construct a $(d+2)$-angulated category starting from a $d$-cluster tilting subcategory, closed under $\Sigma^d$, of a triangulated category with suspension $\Sigma$.

A $d$-cluster tilting subcategory $\mathcal{F}\subseteq \mmod \Phi$ is a functorially finite additive subcategory of $\mmod\Phi$ such that
\begin{align*}
    \mathcal{F}=\{ a\in\mmod\Phi\mid \Ext^{1\,\dots\, d-1} (\mathcal{F},a)=0 \}=\{ a\in\mmod\Phi\mid \Ext^{1\,\dots\, d-1} (a,\mathcal{F})=0 \},
\end{align*}
see Definition \ref{defn_dct}. Suppose such an $\mathcal{F}\subseteq\mmod\Phi$ exists.
Iyama proved in \cite{I} that $\overline{\mathcal{F}}$ is $d$-cluster tilting in $\mathcal{D}^b(\mmod\Phi)$ and so $\overline{\mathcal{F}}$ becomes a $(d+2)$-angulated category by \cite[theorem 1]{GKO}. Note that the $d$-abelian category $\mathcal{F}$ plays the role of a higher $\mmod \Phi$ and $\overline{\mathcal{F}}$ of a higher derived category of $\mathcal{F}$.

Keeping in mind the above, we generalise the theorem by Br\"uning to higher homological algebra as follows.

{\bf Theorem A.}
{\em
There is a bijection
\begin{align*}
    \Bigg\{
    \begin{matrix} \text{functorially finite}\\ \text{wide subcategories} \\ \text{of } \mathcal{F} \end{matrix}
    \Bigg\}
    \rightarrow
    \Bigg\{
    \begin{matrix} \text{functorially finite}\\ \text{wide subcategories} \\ \text{of } \overline{\mathcal{F}} \end{matrix}
    \Bigg\}
    \end{align*}
    sending a wide subcategory $\mathcal{W}$ of $\mathcal{F}$ to $\overline{\mathcal{W}}$.
}

In the above, by a wide subcategory of a $d$-abelian category, we mean an additive subcategory closed under $d$-kernels and $d$-cokernels, and such that every $d$-exact sequence with end terms in the subcategory is Yoneda equivalent to a $d$-exact sequence having all terms in the subcategory. By a wide subcategory of a $(d+2)$-angulated category with automorphism $\Sigma^d$, we mean an additive subcategory closed under $d$-extensions and $\Sigma^{\pm d}$.

Iyama and Yoshino defined Auslander-Reiten $(d+2)$-angles in $(d+2)$-angulated categories in \cite{IY}.
Here, we define Auslander-Reiten $(d+2)$-angles in additive subcategories of $(d+2)$-angulated categories closed under $d$-extensions, an example of which are wide subcategories. We generalise the theorem by J\o rgensen as follows.

{\bf Theorem B.}
{\em
Let $\mathcal{M}$ be a skeletally small $\Hom$-finite $k$-linear $(d+2)$-angulated category with split idempotents. Let $\mathcal{W}$ be an additive subcategory of $\mathcal{M}$ closed under $d$-extensions.

Let $W$ be in $\mathcal{W}$ and suppose that there exists $U^0$ in $\mathcal{W}$ and a non-zero morphism $\gamma^{d+1}:W\rightarrow \Sigma^d U^0$. Let
\begin{align*}
\xymatrix {
\epsilon:& X^0\ar[r]^{\xi^0} & X^1\ar[r]^{\xi^1}& X^2\ar[r] &\cdots\ar[r] & X^d\ar[r]^{\xi^d} & W\ar[r]^{\xi^{d+1}} &\Sigma^d X^0
}
\end{align*}
be an Auslander-Reiten $(d+2)$-angle in $\mathcal{M}$. Then the following are equivalent:
\begin{enumerate}[label=(\alph*)]
    \item $X^0$ has a $\mathcal{W}$-cover of the form $\varphi:W^0\rightarrow X^0$,
    \item there is an Auslander-Reiten $(d+2)$-angle in $\mathcal{W}$ of the form
    \begin{align*}
\xymatrix {
\epsilon':& W^0\ar[r]^{\omega^0} & W^1\ar[r]^{\omega^1}& W^2\ar[r] &\cdots\ar[r] & W^d\ar[r]^{\omega^d} & W\ar[r]^-{\omega^{d+1}} &\Sigma^d W^0.
}
\end{align*}
\end{enumerate}
}

Note that for $d=1$, the above becomes exactly the theorem by J\o rgensen.

{\bf Remark C.}
We will apply Theorems A and B to the class of examples introduced in \cite[section\ 4]{V} and \cite[section\ 7]{HJV}. For positive integers $m,\, l$ and $d$ such that $(m-1)/l =d/2$, consider 
\begin{align*}
\Phi=k A_m/ (\rad_{kA_m})^l.
\end{align*}
There is a unique $d$-cluster tilting subcategory $\mathcal{F}$ of $\mmod\Phi$ with Auslander-Reiten quiver
\begin{align*}
    \xymatrix{
    f_1\ar[r]& f_2\ar[r]&\cdots\ar[r]& f_l\ar[r]&\cdots\ar[r]&f_m\ar[r]&\cdots\ar[r]&f_{m+l-2}\ar[r]& f_{m+l-1},
    }
\end{align*}
where $f_1,\dots, f_m$ are the indecomposable projectives and $f_l,\dots, f_{m+l-1}$ the indecomposable injectives in $\mmod \Phi$. The wide subcategories of $\mathcal{F}$ are fully described in \cite{HJV}.
We consider the $(d+2)$-angulated category $\overline{\mathcal{F}}$. We give a full description of the wide subcategories of $\overline{\mathcal{F}}$, using Theorem A, and a recipe to construct Auslander-Reiten $(d+2)$-angles in these subcategories, using Theorem B.

The paper is organised as follows. Section \ref{section_defn} recalls the definitions of $d$-abelian and $(d+2)$-angulated categories. Section \ref{section_properties} consists of some definitions and a collection of well-known results on $(d+2)$-angulated categories. Section \ref{section_wide} proves Theorem A. Section \ref{section_AR} studies Auslander-Reiten $(d+2)$-angles, both in the ambient category $\mathcal{M}$ and in its subcategory $\mathcal{W}$. Section \ref{section_B} proves Theorem B. Finally, Section \ref{section_example} is an application of Theorems A and B to the class of examples from \cite{V}.

\section{Definitions of $d$-abelian and $(d+2)$-angulated categories}\label{section_defn}

As mentioned in the introduction, if there is a $d$-cluster tilting subcategory $\mathcal{F}\subseteq\mmod\Phi$, then $\mathcal{F}$ is a $d$-abelian category and $\overline{\mathcal{F}}$ is a $(d+2)$-angulated category. In this section we recall the definitions of $d$-abelian and $(d+2)$-angulated categories.

\begin{defn}[{\cite[definitions 2.2 and 2.4]{JG}}]
Let $\mathcal{A}$ be an additive category.
\begin{enumerate}[label=(\alph*)]
    \item A diagram of the form
    $\xymatrix{
    A^0\ar[r]& A^1\ar[r]& A^2\ar[r]&\cdots\ar[r]&A^{d-1}\ar[r]&A^d
    }$
    is a \textit{$d$-kernel} of a morphism $\xymatrix{A^d\ar[r]& A^{d+1}}$ if 
    \begin{align*}
    \xymatrix{
    0\ar[r] & A^0\ar[r]& A^1\ar[r]& A^2\ar[r]&\cdots\ar[r]&A^{d-1}\ar[r]&A^d\ar[r]& A^{d+1}
    }
    \end{align*}
    becomes an exact sequence under $\Hom_{\mathcal{A}}(B,-)$ for each $B$ in $\mathcal{A}$.
    \item A diagram of the form
    $\xymatrix{
    A^1\ar[r]&A^2\ar[r]&\cdots\ar[r]&A^{d-1}\ar[r]&A^d\ar[r]& A^{d+1}
    }$
   is a \textit{$d$-cokernel} of a morphism $\xymatrix{A^0\ar[r]& A^{1}}$ if 
    \begin{align*}
     \xymatrix{
    A^0\ar[r]& A^1\ar[r]& A^2\ar[r]&\cdots\ar[r]&A^{d-1}\ar[r]&A^d\ar[r]& A^{d+1}\ar[r]&0
    }
    \end{align*}
    becomes an exact sequence under $\Hom_{\mathcal{A}}(-,B)$ for each $B$ in $\mathcal{A}$.
    \item A $d$-exact sequence is a diagram of the form
    \begin{align*}
        \xymatrix{
    0\ar[r]&A^0\ar[r]^{\alpha^0}& A^1\ar[r]& A^2\ar[r]&\cdots\ar[r]&A^{d-1}\ar[r]&A^d\ar[r]^{\alpha^d}& A^{d+1}\ar[r]&0,
    }
    \end{align*}
    such that $\xymatrix{A^0\ar[r]^{\alpha^0}& A^1\ar[r]& A^2\ar[r]&\cdots\ar[r]&A^{d-1}\ar[r]&A^d}$ is a $d$-kernel of $\alpha^d$ and $\xymatrix{A^1\ar[r]& A^2\ar[r]&\cdots\ar[r]&A^{d-1}\ar[r]&A^d\ar[r]^{\alpha^d}& A^{d+1}}$ is a $d$-cokernel of $\alpha^0$.
\end{enumerate}
\end{defn}

\begin{defn}[{\cite[definition 3.1]{JG}}]
A \textit{$d$-abelian category} is an additive category $\mathcal{A}$ which satisfies the following axioms:
\begin{enumerate}
    \item[(A0)] The category $\mathcal{A}$ has split idempotents.
    \item[(A1)] Each morphism in $\mathcal{A}$ has a $d$-kernel and a $d$-cokernel.
    \item[(A2)] If $\alpha^0:\xymatrix{A^0\ar[r]&A^1}$ is a monomorphism and
    $\xymatrix{
    A^1\ar[r]&A^2\ar[r]&\cdots\ar[r]& A^{d+1}
    }$
    is a $d$-cokernel of $\alpha^0$, then
    \begin{align*}
        \xymatrix{
    0\ar[r]&A^0\ar[r]^{\alpha^0}& A^1\ar[r]& A^2\ar[r]&\cdots\ar[r]&A^{d-1}\ar[r]&A^d\ar[r]& A^{d+1}\ar[r]&0,
    }
    \end{align*}
    is a $d$-exact sequence.
    \item[(A2$^{\text{op}}$)] If $\alpha^d:\xymatrix{A^d\ar[r]&A^{d+1}}$ is an epimorphism and
    $\xymatrix{
    A^0\ar[r]&\cdots\ar[r]& A^{d-1}\ar[r]& A^{d}
    }$
    is a $d$-kernel of $\alpha^d$, then
    \begin{align*}
        \xymatrix{
    0\ar[r]&A^0\ar[r]& A^1\ar[r]& A^2\ar[r]&\cdots\ar[r]&A^{d-1}\ar[r]&A^d\ar[r]^{\alpha^d}& A^{d+1}\ar[r]&0,
    }
    \end{align*}
    is a $d$-exact sequence.
\end{enumerate}
\end{defn}

\begin{defn}[{\cite[definition\ 1.1]{GKO}}]
Let $\mathcal{M}$ be an additive category and let $\Sigma^d$ be an automorphism of $\mathcal{M}$ with inverse $\Sigma^{-d}$. A \textit{$\Sigma^d$-sequence} is a sequence of morphisms in $\mathcal{M}$ of the form
\begin{align}\label{diagram_angle}
\xymatrix {
 \epsilon : & X^0\ar[r]^{\xi^0} & X^1\ar[r]^{\xi^1}& X^2\ar[r] &\cdots\ar[r] & X^d\ar[r]^{\xi^d} & X^{d+1}\ar[r]^{\xi^{d+1}} &\Sigma^d X^0.
}
\end{align}
A \textit{morphism of $\Sigma^d$-sequences} is given by a sequence of morphisms $\varphi=(\varphi^0,\dots,\,\varphi^{d+1})$ such that the following diagram commutes:
\begin{align*}
    \xymatrix {
\epsilon:\ar[d]^\varphi &X^0\ar[r]^{\xi^0}\ar[d]^{\varphi^0} & X^1\ar[r]^{\xi^1}\ar[d]^{\varphi^1}& X^2\ar[r]\ar[d]^{\varphi^2} &\cdots\ar[r] & X^d\ar[r]^{\xi^d}\ar[d]^{\varphi^d} & X^{d+1}\ar[r]^{\xi^{d+1}}\ar[d]^{\varphi^{d+1}} &\Sigma^d X^0\ar[d]^{\Sigma^d \varphi^0}\\
\epsilon': &Y^0\ar[r]_{\eta^0} & Y^1\ar[r]_{\eta^1}& Y^2\ar[r] &\cdots\ar[r] & Y^d\ar[r]_{\eta^d} & Y^{d+1}\ar[r]_{\eta^{d+1}} &\Sigma^d Y^0.
}
\end{align*}
\end{defn}

\begin{defn}[{\cite[definition\ 1.1]{GKO}}]\label{defn_angles}
A $(d+2)$\textit{-angulated category} is a triple $(\mathcal{M},\Sigma^d,\pentagon)$, where $\mathcal{M}$, $\Sigma^d$ are as above and $\pentagon$ is a collection of $\Sigma^d$-sequences, called \textit{$(d+2)$-angles}, satisfying the following axioms.
\begin{enumerate}
    \item[(N1)] The collection $\pentagon$ is closed under sums and summands and, for every $X\in\mathcal{M}$, the \textit{trivial $\Sigma^d$-sequence}
    \begin{align*}
    \xymatrix {
    \epsilon : & X\ar[r]^{1_{X}} & X\ar[r]& 0\ar[r] &\cdots\ar[r] & 0\ar[r] & 0\ar[r] &\Sigma^d X}
    \end{align*}
    is in $\pentagon$.
    For each morphism $\xi^0:X^0\rightarrow X^1$ in $\mathcal{M}$, there is a $(d+2)$-angle in $\pentagon$ of the form (\ref{diagram_angle}).
    
    \item[(N2)] A $\Sigma^d$-sequence (\ref{diagram_angle}) is in $\pentagon$ if and only if so is its \textit{left rotation}:
    \begin{align*}
    \xymatrix {
    X^1\ar[r]^{\xi^1}& X^2\ar[r] &\cdots\ar[r] & X^d\ar[r]^{\xi^d} & X^{d+1}\ar[r]^{\xi^{d+1}} &\Sigma^d X^0\ar[rr]^{(-1)^d\Sigma^d(\xi^0)} && \Sigma^d X^1.
    }
    \end{align*}
    
    \item[(N3)] Each commutative diagram of solid arrows, with rows in $\pentagon$
    \begin{align*}
    \xymatrix {
   X^0\ar[r]^{\xi^0}\ar[d]^{\varphi^0} & X^1\ar[r]^{\xi^1}\ar[d]^{\varphi^1}& X^2\ar[r]\ar@{-->}[d]^{\varphi^2} &\cdots\ar[r] & X^d\ar[r]^{\xi^d}\ar@{-->}[d]^{\varphi^d} & X^{d+1}\ar[r]^{\xi^{d+1}}\ar@{-->}[d]^{\varphi^{d+1}} &\Sigma^d X^0\ar[d]^{\Sigma^d \varphi^0}\\
    Y^0\ar[r]_{\eta^0} & Y^1\ar[r]_{\eta^1}& Y^2\ar[r] &\cdots\ar[r] & Y^d\ar[r]_{\eta^d} & Y^{d+1}\ar[r]_{\eta^{d+1}} &\Sigma^d Y^0,
    }
    \end{align*}
    can be completed as indicated to a morphism of $\Sigma^d$-sequences.
    
    \item[(N4)] In the situation of (N3), the morphisms $\varphi^2,\dots,\,\varphi^{d+1}$ can be chosen such that
   \begin{align*}
       \xymatrix{
       X^1\oplus Y^0\ar[rr]^{\begin{psmallmatrix} -\xi^1 &0\\ \varphi^1 &\eta^0 \end{psmallmatrix}} &&X^2\oplus Y^1\ar[r]& \cdots\ar[r]& \Sigma^d X^0\oplus Y^{d+1}\ar[rr]^{\begin{psmallmatrix} -\Sigma^d\xi^0 &0\\ \Sigma\varphi^0 &\eta^{d+1} \end{psmallmatrix}}&& \Sigma^d X^1 \oplus \Sigma^d Y^0
       }
   \end{align*}
   belongs to $\pentagon$.
\end{enumerate}
\end{defn}

\section{Properties of $(d+2)$-angulated categories}\label{section_properties}
In this section, we present the setup we will be working in. Then, we introduce some terminology that will be used in later sections and we state some well known properties of $(d+2)$-angulated categories.

\begin{defn}
Let $\mathcal{A}$ be an additive category. An \textit{additive subcategory of $\mathcal{A}$} is a full subcategory which is closed under direct sums, direct summands and isomorphisms in $\mathcal{A}$.
\end{defn}

\begin{setup}\label{setup}
Let $\mathcal{M}$ be a skeletally small $k$-linear $\Hom$-finite $(d+2)$-angulated category with split idempotents. Note that this implies that $\mathcal{M}$ is Krull-Schmidt by \cite[p. 52, theorem of Krull-Schmidt]{R}.

Let $\mathcal{W}$ be an additive subcategory of $\mathcal{M}$ closed under $d$-extensions in the sense that given any morphism in $\mathcal{M}$ of the form $\delta: W''\rightarrow \Sigma^d W'$ with $W',\,W''\in\mathcal{W}$, there is a $(d+2)$-angle in $\mathcal{M}$ of the form
\begin{align*}
\xymatrix {
W'\ar[r] & W^1\ar[r] &\cdots\ar[r] & W^d\ar[r] & W''\ar[r]^-{\delta} &\Sigma^d W'
}
\end{align*}
with $W^i\in\mathcal{W}$ for any $i\in\{1,\dots,\, d\}$.
\end{setup}

\begin{defn}
An additive subcategory $\mathcal{W}$ of a $(d+2)$-angulated category $\mathcal{M}$ is called \textit{wide} if it is closed under $d$-extensions and satisfies $\Sigma^{d}(\mathcal{W})\subseteq \mathcal{W}$ and $\Sigma^{-d}(\mathcal{W})\subseteq \mathcal{W}$.
\end{defn}

\begin{defn}[{\cite[definition\ 1.1, chapter IV]{A}}]
Consider an additive category $\mathcal{X}$ and let $X^0,\,X^1,\,X^d,\,X^{d+1}\in\mathcal{X}$. We say that
\begin{enumerate}[label=(\alph*)]
    \item a morphism $\xi^0:X^0\rightarrow X^1$ is \textit{left almost split in $\mathcal{X}$} if it is not a split monomorphism and for every $Y$ in $\mathcal{X}$, every morphism $\gamma:X^0\rightarrow Y$ which is not a split monomorphism factors through $\xi^0$, \textbf{i.e.} there exists a morphism $X^1\rightarrow Y$ such that the following diagram commutes:
    \begin{align*}
        \xymatrix  {
        X^0\ar[rr]^{\xi^0}\ar[rd]_\gamma&& X^1;\ar@{-->}[ld]^{\exists}\\
        & Y &
        }
    \end{align*}
    \item a morphism $\xi^d:X^d\rightarrow X^{d+1}$ is \textit{right almost split in $\mathcal{X}$} if it is not a split epimorphism and for every $Y$ in $\mathcal{X}$, every morphism $\delta:Y\rightarrow X^{d+1}$ which is not a split epimorphism factors through $\xi^d$, \textbf{i.e.} there exists a morphism $Y\rightarrow X^d$ such that the following diagram commutes:
    \begin{align*}
        \xymatrix  {
        X^d\ar[rr]^{\xi^d}&& X^{d+1}.\\
        & Y\ar[ru]_\delta \ar@{-->}[lu]^{\exists} &
        }
    \end{align*}
\end{enumerate}
\end{defn}

\begin{defn}[{\cite[definition\ 1.1, chapter IV]{A}}]
A morphism $\xi:X\rightarrow Y$ is \textit{right minimal} if each morphism $\varphi:X\rightarrow X$ which satisfies $\xi\circ\varphi=\xi$ is an isomorphism. Dually, $\xi$ is \textit{left minimal} if each morphism $\nu:Y\rightarrow Y$ which satisfies $\nu\circ\xi=\xi$ is an isomorphism.
\end{defn}

\begin{defn}
[{\cite[definition \ 1.4]{JP}}]
Let $X\in\mathcal{M}$. A $\mathcal{W}$\textit{-precover} of $X$ is a morphism of the form $\omega :W\rightarrow X$ with $W\in \mathcal{W}$ such that every morphism $\omega':W'\rightarrow X$ with $W'\in \mathcal{W}$ factorizes as:
\begin{align*}
\xymatrix{
W'\ar[rr]^{\omega'} \ar@{-->}[dr]_{\exists}& & X.\\
& W \ar[ru]_{\omega} &
}
\end{align*}
A $\mathcal{W}$\textit{-cover} of $X$ is a $\mathcal{W}$-precover of $X$ which is also a right minimal morphism.
The dual notions of precovers and covers are \textit{preenvelopes} and \textit{envelopes} respectively.
\end{defn}

\begin{defn}
The subcategory $\mathcal{W}$ of $\mathcal{M}$ is called \textit{precovering} if every object in $\mathcal{M}$ has a $\mathcal{W}$-precover. Dually, $\mathcal{W}$ is \textit{preenveloping} if every object in $\mathcal{M}$ has a $\mathcal{W}$-preenvelope. If $\mathcal{W}$ is both precovering and preenveloping, we say that it is \textit{functorially finite}.
\end{defn}

\begin{remark}
Note that by \cite[proposition 2.5(a)]{GKO}, any $(d+2)$-angle in $\mathcal{M}$ of the form
\begin{align*}
\xymatrix {
 X^0\ar[r] & X^1\ar[r]& X^2\ar[r] &\cdots\ar[r] & X^d\ar[r] & X^{d+1}\ar[r] &\Sigma^d X^0,
}
\end{align*}
is such that the induced sequence 
\begin{align*}
    \xymatrix @C=1em{
 \cdots\ar[r]&\mathcal{M}(-,\Sigma^{-d}X^{d+1})\ar[r]&\mathcal{M}(-,X^0)\ar[r] &\cdots\ar[r] &  \mathcal{M}(-,X^{d+1})\ar[r] &\mathcal{M}(-,\Sigma^d X^0)\ar[r]&\cdots
}
\end{align*}
is exact. The following three lemmas are direct consequences of this result.
\end{remark}

\begin{lemma}\label{lemma_consecutive}
Any two consecutive morphisms in a $(d+2)$-angle compose to zero.
\end{lemma}

\begin{lemma}\label{lemma_zeroiffexists1}
Consider
\begin{align*}
\xymatrix {
&& A\ar[d]^\alpha\ar@{-->}[ld]\ar[rd]^0 \\
\epsilon: & X^0\ar[r]_{\xi^0} & X^1\ar[r]_{\xi^{1}}& X^2\ar[r] &\cdots\ar[r] & X^{d+1}\ar[r]_{\xi^{d+1}}&\Sigma^d X^0,
}
\end{align*}
where $\epsilon$ is a $(d+2)$-angle in $\mathcal{M}$. Then $\xi^1\circ \alpha=0$ if and only if there exists a morphism $\delta:A\rightarrow X^0$ such that $\xi^0\circ \delta =\alpha$.
\end{lemma}

\begin{lemma}\label{lemma_zeroiffexists2}
Consider
\begin{align*}
\xymatrix {
\epsilon: & X^0\ar[r]^{\xi^0} & X^1\ar[r] &\cdots\ar[r] & X^d\ar[r]^{\xi^d}\ar[rd]_0 & X^{d+1}\ar[r]^{\xi^{d+1}}\ar[d]_\varphi &\Sigma^d X^0,\ar@{-->}[ld]\\
&&&&& A
}
\end{align*}
where $\epsilon$ is a $(d+2)$-angle in $\mathcal{M}$. Then $\varphi\circ \xi^d=0$ if and only if there exists a morphism $\phi:\Sigma^d X^0\rightarrow A$ such that $\phi\circ \xi^{d+1}=\varphi$.
\end{lemma}

The following two lemmas are well-known in the triangulated case. Similar proofs apply to the $(d+2)$-angulated case.

\begin{lemma}\label{lemma_lrmin}
Consider a $(d+2)$-angle of the form
\begin{align*}
\xymatrix {
 X^0\ar[r]^{\xi^0} & X^1\ar[r]^{\xi^1}& X^2\ar[r] &\cdots\ar[r] & X^d\ar[r]^{\xi^d} & X^{d+1}\ar[r]^{\xi^{d+1}} &\Sigma^d X^0.
}
\end{align*}
Then:
\begin{enumerate} [label=(\alph*)]
\item $\xi^1$ is right minimal if and only if $\xi^0\in \rad_\mathcal{M}$,
\item $\xi^d$ is left minimal if and only if $\xi^{d+1}\in \rad_\mathcal{M}$.
\end{enumerate}
\end{lemma}

%It is easy to prove the following lemma using Lemma \ref{lemma_lrmin}. This is the generalisation of \cite[lemma\ 2.5]{KH} to $(d+2)$-angles. (MAYBE CANCEL following lemma)
%\begin{lemma}
%Consider a $(d+2)$-angle of the form
%\begin{align*}
%\xymatrix {
 %X^0\ar[r]^{\xi^0} & X^1\ar[r]^{\xi^1}& X^2\ar[r] &\cdots\ar[r] & X^d\ar[r]^{\xi^d} & %X^{d+1}\ar[r]^{\xi^{d+1}} &\Sigma^d X^0.
%}
%\end{align*}
%The following are equivalent
%\begin{enumerate} [label=(\alph*)]
 %   \item $\xi^1$ is right minimal,
  %  \item $\xi^{d+1}$ is left minimal,
   % \item $\xi^0\in \rad_\mathcal{M}$. 
%\end{enumerate}
%\end{lemma}

\begin{lemma}\label{lemma_epimono}
Consider a $(d+2)$-angle of the form
\begin{align*}
\xymatrix {
 X^0\ar[r]^{\xi^0} & X^1\ar[r]^{\xi^1}& X^2\ar[r] &\cdots\ar[r] & X^d\ar[r]^{\xi^d} & X^{d+1}\ar[r]^{\xi^{d+1}} &\Sigma^d X^0.
}
\end{align*}
The following are equivalent:
\begin{enumerate} [label=(\alph*)]
\item $\xi^{d+1}=0$,
\item $\xi^d$ is a split epimorphism,
\item $\xi^0$ is a split monomorphism.
\end{enumerate}
\end{lemma}

In the case $d=1$, so in the case of a triangulated category, a morphism can be extended to a triangle in a unique way up to isomorphism. On the other hand, for $d>1$, a morphism can be extendend to a $(d+2)$-angle in different non-isomorphic ways. However, we still have a unique ``minimal" $(d+2)$-angle extending any given morphism.

%\begin{lemma}
%A $(d+2)$-angle of the form
%\begin{align*}
%\xymatrix {
%A\oplus X^0\ar[rr]^{\begin{psmallmatrix}1_A & \beta \\ \alpha &\xi^0\end{psmallmatrix}} && A\oplus X^1\ar[r]^-{\xi^1}& X^2\ar[r] &\cdots\ar[r] & X^d\ar[r]^{\xi^d} & X^{d+1}\ar[r]^{\xi^{d+1}} &\Sigma^d X^0
%}
%\end{align*}
%is isomorphic to the direct sum of the two $(d+2)$-angles
%\begin{align*}
%\xymatrix {
% A\ar[r]^{1_A} & A\ar[r]& 0\ar[r] &\cdots\ar[r] &0\ar[r] & 0\ar[r] &\Sigma^d A & \text{and} \\
% X^0\ar[r]^{\overline{\xi^0}} & X^1\ar[r]^{\xi^1}& X^2\ar[r] &\cdots\ar[r] & X^d\ar[r]^{\xi^d} & X^{d+1}\ar[r]^{\xi^{d+1}} &\Sigma^d X^0, &
%}
%\end{align*}
%where $\overline{\xi^0}=-\alpha\beta+\xi^0$.
%\end{lemma}

%\begin{lemma}
%Let $\delta: M''\rightarrow \Sigma^d M'$ be any morphism in $\mathcal{M}$ and consider a $(d+2)$-angle extending it:
%\begin{align*}
%\xymatrix {
% M'\ar[r]^{\xi^0} & X^1\ar[r]^{\xi^1}& X^2\ar[r] &\cdots\ar[r] & X^{d-1}\ar[r]^{\xi^{d-1}} & X^d\ar[r]^{\xi^d} & M''\ar[r]^-{\delta} &\Sigma^d M'.
%}
%\end{align*}
%Then $\xi^1,\,\dots,\, \xi^{d-1}$ are in $\rad_\mathcal{M}$ if and only if for $i=1,\,\dots,\,d-1$ there is no $A$ in $\mathcal{M}$ such that 
%\begin{align*}
%    \xi^i\cong \begin{pmatrix} 1_A & \beta\\\alpha& \overline{\xi^i} \end{pmatrix}: A\oplus \overline{X^i}\rightarrow A\oplus \overline{X^{i+1}}.
%\end{align*}
%\end{lemma}

\begin{lemma}[{\cite[lemma\ 5.18]{OT}}] \label{lemma_uni_min}
Let $d>1$ and $\delta: M''\rightarrow \Sigma^d M'$ be any morphism in $\mathcal{M}$. Then, up to isomorphism, there exists a unique $(d+2)$-angle of the form
\begin{align*}
\xymatrix {
 M'\ar[r]^{\xi^0} & X^1\ar[r]^{\xi^1}& X^2\ar[r] &\cdots\ar[r] & X^{d-1}\ar[r]^-{\xi^{d-1}} & X^d\ar[r]^{\xi^d} & M''\ar[r]^-{\delta} &\Sigma^d M',
}
\end{align*}
with $\xi^1, \dots,\, \xi^{d-1}$ in $\rad_\mathcal{M}$.
\end{lemma}

\begin{remark}
Note that when $d>1$, for a $(d+2)$-angle in $\mathcal{M}$ of the form
\begin{align*}
\xymatrix {
 W'\ar[r]^{\xi^0} & X^1\ar[r]^{\xi^1}& X^2\ar[r] &\cdots\ar[r] & X^{d-1}\ar[r]^-{\xi^{d-1}} & X^d\ar[r]^{\xi^d} & W''\ar[r]^-{\delta} &\Sigma^d W',
}
\end{align*}
with $W',\,W''$ in $\mathcal{W}$, it is not necessarily true that $X^1,\dots,\, X^d$ are in $\mathcal{W}$. However, if $\xi^1,\dots,\,\xi^{d-1}$ are in $\rad_{\mathcal{M}}$, then $X^1,\dots,\, X^d$ are in $\mathcal{W}$.
\end{remark}

\section{Proof of Theorem A}\label{section_wide}
The aim of this section is to prove Theorem A. We work in the setup from the introduction, assuming there is a $d$-cluster tilting subcategory $\mathcal{F}\subseteq \mmod\Phi$. Recall that we assume that the global dimension of $\Phi$ is at most $d$.
Note that by \cite[theorem\ 1.21]{I}, $\overline{\mathcal{F}}$ is $d$-cluster tilting in $\mathcal{D}^b(\mmod\Phi)$ and so, by \cite[theorem\ 1]{GKO}, it follows that it is $(d+2)$-angulated.

We start by recalling the definition of $d$-cluster tilting subcategories of $\mmod\Phi$ and $\mathcal{D}^b(\mmod\Phi)$.

%\begin{defn}
%{\cite[section\ 3.3]{JG}}
%Let $\mathcal{A}$ be a category and $\mathcal{C}\subseteq \mathcal{A}$ a full subcategory. We say that $\mathcal{C}$ is \textit{generating} if for each $a\in\mathcal{A}$ there is an epimorphism $c\twoheadrightarrow a$ with $c$ in $\mathcal{C}$. We say that $\mathcal{C}$ is \textit{cogenerating} if for each $a\in\mathcal{A}$ there is a monomorphism $a\hookrightarrow c$ with $c$ in  $\mathcal{C}$.
%\end{defn}

\begin{defn}
[{{\cite[definition\ 2.2]{IO},} {\cite[definition\ 3.14]{JG},} {\cite[definition\ 3.3]{JL}}}]\label{defn_dct}
Let $\mathcal{A}$ be either $\mmod\Phi$ or $\mathcal{D}^b(\mmod\Phi)$ and $\mathcal{C}$ be a full subcategory  of $\mathcal{A}$. We say that $\mathcal{C}$ is a \textit{$d$-cluster tilting subcategory of $\mathcal{A}$} if:
\begin{enumerate}[label=(\alph*)]
    \item $\mathcal{C}=\{ a\in\mathcal{A}\mid \Ext^{1\,\dots\, d-1}_\mathcal{A} (\mathcal{C},a)=0 \}=\{ a\in\mathcal{A}\mid \Ext^{1\,\dots\, d-1}_\mathcal{A} (a,\mathcal{C})=0 \}$,
    \item $\mathcal{C}$ is functorially finite.
\end{enumerate}
\end{defn}

We build the proof of Theorem A by first proving a more general bijection, then proving this bijection respects ``functorially finite''. Proving Theorem A will then amount to proving the bijection respects ``wide''.

\begin{lemma}\label{lemma_bij_gen}
There is a bijection
\begin{align*}
    \Bigg\{
    \begin{matrix} \text{additive subcategories}\\  \text{of } \mathcal{F} \end{matrix}
    \Bigg\}
    \rightarrow
    \Bigg\{
    \begin{matrix} \text{additive subcategories} \\ \text{of } \overline{\mathcal{F}} \\ \text{closed under } \Sigma^{\pm d} \end{matrix}
    \Bigg\}
    \end{align*}
sending an additive subcategory $\mathcal{W}$ of $\mathcal{F}$ to $\overline{\mathcal{W}}$.
\end{lemma}

\begin{proof}
Let $\mathcal{W}\subseteq \mathcal{F}$ be an additive subcategory of $\mathcal{F}$, then $\overline{\mathcal{W}}\subseteq \overline{\mathcal{F}}$ is clearly additive and closed under $\Sigma^{\pm d}$.

Suppose now that $\mathcal{X}\subseteq\overline{\mathcal{F}}$ is an additive subcategory closed under $\Sigma^{\pm d}$.
Let $x$ be an indecomposable in $\mathcal{X}$, then $x=\Sigma^{id} f$ for some $f\in\mathcal{F}$ and  integer $i$. Since $\mathcal{X}$ is closed under $\Sigma^{\pm d}$ and under direct summands, then 
\begin{align*}
\add \{ \Sigma^{id}f\mid i\in\mathbb{Z} \}\subseteq \mathcal{X}.
\end{align*}
Take $\mathcal{W}:=\mathcal{F}\cap\mathcal{X}$ and note that, by the above, we have $\add\{ \Sigma^{id}\mathcal{W}\mid i\in\mathbb{Z} \}\subseteq \mathcal{X}$. Moreover, if $x$ is an indecomposable in $\mathcal{X}$, say $x=\Sigma^{id}f$, then $\Sigma^{(-i)d} x=f\in\mathcal{X}$ and so $f\in\mathcal{W}$. Hence $x\in \add\{ \Sigma^{id}\mathcal{W}\mid i\in\mathbb{Z} \}$ and $\overline{\mathcal{W}}=\add\{ \Sigma^{id}\mathcal{W}\mid i\in\mathbb{Z} \}=\mathcal{X}$.
\end{proof}

\begin{lemma}\label{lemma_func_fin}
The bijection from Lemma \ref{lemma_bij_gen} respects ``functorially finite''.
%Let $\mathcal{W}\subseteq \mathcal{F}$ be an additive subcategory and $\overline{\mathcal{W}}:=\add\{ \Sigma^{id} \mathcal{W}\mid i\in \mathbb{Z} \}\subseteq \overline{\mathcal{F}}$. Then $\mathcal{W}$ is functorially finite in $\mathcal{F}$ if and only if $\overline{\mathcal{W}}$ is functorially finite in $\overline{\mathcal{F}}$. 
\end{lemma}

\begin{proof}
Suppose first that $\overline{\mathcal{W}}$ is functorially finite in $\overline{\mathcal{F}}$ and take any $f\in\mathcal{F}$. Then, there is a $\overline{\mathcal{W}}$-precover of $f$ of the form $\overline{\omega}: \overline{w}\rightarrow f$. Since $f$ is concentrated in degree zero, then we may assume $\overline{w}=w_0\oplus \Sigma^{-d} w_{-d}$ for some $w_0,\, w_{-d}\in\mathcal{W}$, as any other summand of $\overline{w}$ would have zero $\Hom$ space to $f$. Then,
\begin{align*}
    \overline{\omega}=(\omega_0,\,\omega_{-d}): w_0\oplus \Sigma^{-d} w_{-d}\rightarrow f.
\end{align*}
Let $w\in\mathcal{W}$ and $\alpha: w\rightarrow f$. Since $\mathcal{W}\subseteq\overline{\mathcal{W}}$, there is a morphism $\gamma:w\rightarrow w_0\oplus \Sigma^{-d} w_{-d}$ such that $\overline{\omega}\circ\gamma=\alpha$. Since there are no non-zero maps of the form $w\rightarrow\Sigma^{-d} w_{-d}$, then
\begin{align*}
    \alpha=\overline{\omega}\circ\gamma=(\omega_0,\,\omega_{-d})\circ \begin{pmatrix}\gamma_0\\0\end{pmatrix}=\omega_0\circ\gamma_0.
\end{align*}
Hence $\omega_0$ is a $\mathcal{W}$-precover of $f$ and $\mathcal{W}$ is precovering in $\mathcal{F}$. Dually,  $\mathcal{W}$ is preenveloping in $\mathcal{F}$.

Suppose now that $\mathcal{W}$ is functorially finite in $\mathcal{F}$. Note that, in order to prove that $\overline{\mathcal{W}}$ is precovering in $\overline{\mathcal{F}}$, it is enough to find a $\overline{\mathcal{W}}$-precover of any $f\in\mathcal{F}$.
We have that $\mathcal{W}\subseteq\mathcal{F}$ is functorially finite, $\mathcal{F}\subseteq \mmod\Phi$ is functorially finite since $\mathcal{F}$ is $d$-cluster tilting in $\mmod\Phi$ and $\mmod\Phi\subseteq \mathcal{D}^b(\mmod\Phi)$ is functorially finite by \cite[theorem\ 5.1]{I}. Hence $\mathcal{W}\subseteq\mathcal{D}^b(\mmod\Phi)$ is functorially finite. Moreover, for any integer $i$, applying the automorphism $\Sigma^i$ to
\begin{align*}
    \mathcal{W}\subseteq\mathcal{F}\subseteq \mmod\Phi\subseteq \mathcal{D}^b(\mmod\Phi),
\end{align*}
we conclude that $\Sigma^i\mathcal{W}\subseteq\mathcal{D}^b(\mmod\Phi)$ is functorially finite. For $f\in\mathcal{F}$, note that the only non-zero morphisms from $\overline{\mathcal{W}}$ to $f$ are from objects in $\mathcal{W}\oplus\Sigma^{-d}\mathcal{W}$. Take a $\mathcal{W}$-precover of $f$, say $\omega_0:w_0\rightarrow f$, and a  $\Sigma^{-d}\mathcal{W}$-precover of $f$, say $\omega_{-d}:\Sigma^{-d}w_{-d}\rightarrow f$. Consider
\begin{align*}
    \overline{\omega}:=(\omega_0,\,\omega_{-d}):w_0\oplus \Sigma^{-d}w_{-d}\rightarrow f.
\end{align*}
Given any $\overline{v}$ in $\overline{\mathcal{W}}$ and $\overline{\nu}:\overline{v}\rightarrow f$, without loss of generality, let $\overline{v}=v_0\oplus\Sigma^{-d}v_{-d}$ for some $v_0,\,v_{-d}\in\mathcal{W}$. So
\begin{align*}
    \overline{\nu}=(\nu_0,\,\nu_{-d}):v_0\oplus \Sigma^{-d}v_{-d}\rightarrow f.
\end{align*}
Then, there are $\gamma_0:v_0\rightarrow w_0$ and $\gamma_{-d}:\Sigma^{-d} v_{-d}\rightarrow\Sigma^{-d}w_{-d}$ such that $\nu_0=\omega_0\circ\gamma_0$ and $\nu_{-d}=\omega_{-d}\circ\gamma_{-d}$. Hence
\begin{align*}
    \overline{\omega}\circ\begin{pmatrix} \gamma_0 & 0\\0&\gamma_{-d} \end{pmatrix}=(\omega_0\circ\gamma_0, \,\omega_{-d}\circ\gamma_{-d})=\nu
\end{align*}
and $\overline{\omega}$ is a $\overline{\mathcal{W}}$-precover of $f$. Dually, $\overline{\mathcal{W}}$ is preenveloping in $\overline{\mathcal{F}}$.
\end{proof}

\begin{lemma}\label{lemma_dsums}
Let $\mathcal{A}$ and $\mathcal{B}$ be additive categories and $\mathcal{A}$ have split idempotents. Suppose $F:\mathcal{A}\rightarrow \mathcal{B}$ is a full and faithful additive functor, then $F(\mathcal{A})$ is closed under direct summands. 
\end{lemma}

\begin{proof}
Let $a\in\mathcal{A}$ satisfy $F(a)=x\oplus y$. Then, we have a biproduct diagram:
\begin{align}\label{diagram_biproduct}
    \xymatrix{
    x\ar@/^/[rr]^i&& F(a)\ar@/^/[ll]^p\ar@/_/[rr]_q && y, \ar@/_/[ll]_j
    }
\end{align}
where $pi=1_x$ and $qj=1_y$. Also, $e=ip$ and $1-e=jq$ are idempotents in $\End_\mathcal{B}(F(a))$.
Now, as $F$ is full and faithful, there is an idempotent $e'$ in $\End_\mathcal{A}(a)$ such that $F(e')=e$, and $F(1-e')=1-e$. Since $\mathcal{A}$ has split idempotents, we get a biproduct diagram:
\begin{align*}
    \xymatrix{
    x'\ar@/^/[rr]^{i'}&& a\ar@/^/[ll]^{p'}\ar@/_/[rr]_{q'} && y', \ar@/_/[ll]_{j'}
    }
\end{align*}
where $p'i'=1_{x'}$, $q'j'=1_{y'}$, $i'p'=e'$ and $j'q'=1-e'$. Applying $F$ to this, we get:
\begin{align*}
    \xymatrix{
    F(x')\ar@/^/[rr]^{F(i')}&& F(a)\ar@/^/[ll]^{F(p')}\ar@/_/[rr]_{F(q')} && F(y'). \ar@/_/[ll]_{F(j')}
    }
\end{align*}
We show this is isomorphic to (\ref{diagram_biproduct}). We have
\begin{align*}
    &ip=e=F(i')F(p'),\, pi=1_x,\, F(p')F(i')=F(p'i')=F(1_{x'})=1_{F(x')}, &&\text{ then}\\
    &F(p')i\circ pF(i')= F(p')F(i')F(p')F(i')=1_{F(x')} && \text{ and} \\
    &pF(i')\circ F(p')i=pipi=1_x.
\end{align*}
Hence $x\cong F(x')$. Similarly, $y\cong F(y')$.
\end{proof}

The following is a well-known result, so we do not present a proof here.
\begin{lemma}\label{lemma_d-kernel}
Suppose we have an exact sequence with terms in $\mathcal{F}$ of the form:
\begin{align}\label{diagram_exact_ker}
    \xymatrix{
0\ar[r] & f^{0}\ar[r]^{\varphi^0}&f^1\ar[r]^{\varphi^1}&\cdots\ar[r]^{\varphi^{d-2}}& f^{d-1}\ar[r]^-{\varphi^{d-1}}&f^d\ar[r]^{\varphi^d}& f^{d+1}.
}
\end{align}
Then
\begin{align*}
    \xymatrix{
0\ar[r] & f^{0}\ar[r]^{\varphi^0}&f^1\ar[r]^{\varphi^1}&\cdots\ar[r]^{\varphi^{d-2}}& f^{d-1}\ar[r]^-{\varphi^{d-1}}&f^d
}
\end{align*}
is a $d$-kernel in $\mathcal{F}$ of $\varphi^d$.
\end{lemma}

\begin{lemma}\label{lemma_d+2angulated}
Let $\mathcal{D}$ and $\mathcal{D}'$ be triangulated categories with suspension functors $\Sigma$ and $\Sigma'$ respectively. Suppose there are $d$-cluster tilting subcategories $\mathcal{C}\subseteq \mathcal{D}$ and $\mathcal{C}'\subseteq \mathcal{D}'$ such that $\Sigma^d(\mathcal{C})\subseteq \mathcal{C}$ and $(\Sigma')^d(\mathcal{C}')\subseteq \mathcal{C}'$. Suppose $F:\mathcal{D}\rightarrow\mathcal{D}'$ is a triangulated functor  such that $F(\mathcal{C})\subseteq \mathcal{C'}$. Then $F$ sends $(d+2)$-angles in $\mathcal{C}$ to $(d+2)$-angles in $\mathcal{C}'$.
\end{lemma}

\begin{proof}
First note that $(\mathcal{C}, \Sigma^d)$ and $(\mathcal{C}',(\Sigma')^d)$ are $(d+2)$-angulated categories by \cite[theorem\ 1]{GKO}. Take any $(d+2)$-angle in $\mathcal{C}$, say
\begin{align*}
    \xymatrix{
c^0\ar[r]& c^1\ar[r]&\cdots\ar[r]& c^d\ar[r]&c^{d+1}\ar[r]^\gamma&\Sigma^d c^0.
}
\end{align*}
This comes from a diagram in $\mathcal{D}$ of the form
\begin{align}\label{diagram_angle2}
\xymatrix @C=1em{
& c^1\ar[rd]\ar[rr] && c^2\ar[rd]\ar[r]&&\cdots&\ar[r]& c^{d-1}\ar[rr]\ar[rd] &&c^d\ar[rd]
\\
c^0\ar[ru]&& x^1\ar@{~>}[ll]\ar[ru]&& x^2\ar@{~>}[ll]&\cdots& x^{d-2}\ar[ru]&& x^{d-1} \ar@{~>}[ll]\ar[ru]&& c^{d+1},\ar@{~>}[ll]
}
\end{align}
where by $\xymatrix{x\ar@{~>}[r]&y}$, we mean a morphism $x\rightarrow\Sigma y$ and the composition of all the wavy arrows is $\gamma$. Each oriented triangle is a triangle in $\mathcal{D}$ and each non-oriented triangle is commutative. Applying the functor $F$ to (\ref{diagram_angle2}), we get the diagram:
\begin{align*}
\xymatrix @C=0.004em{
& F(c^1)\ar[rd]\ar[rr] && F(c^2)\ar[rd]\ar[r]&&\cdots&\ar[r]& F(c^{d-1})\ar[rr]\ar[rd] &&F(c^d)\ar[rd]
\\
F(c^0)\ar[ru]&& F(x^1)\ar@{~>}[ll]\ar[ru]&& F(x^2)\ar@{~>}[ll]&\cdots& F(x^{d-2})\ar[ru]&& F(x^{d-1}) \ar@{~>}[ll]\ar[ru]&& F(c^{d+1}),\ar@{~>}[ll]
}
\end{align*}
where each non-oriented triangle is commutative, and since $F$ is triangulated, each oriented triangle is a triangle in $\mathcal{D}'$ and by $\xymatrix{F(x)\ar@{~>}[r]&F(y)}$ we mean a morphism $F(x)\rightarrow F(\Sigma y)=\Sigma' F(y)$. Then, by \cite[theorem\ 1]{GKO}, we obtain a $(d+2)$-angle in $\mathcal{C}'$:
\begin{align*}
    \xymatrix{
F(c^0)\ar[r]& F(c^1)\ar[r]&\cdots\ar[r]& F(c^d)\ar[r]&F(c^{d+1})\ar[r]^-{F(\gamma)}&(\Sigma')^d F(c^0).
}
\end{align*}
\end{proof}

\begin{remark}\label{rmk_angle_seq}
Note that any $d$-exact sequence in $\mathcal{F}$ induces a $(d+2)$-angle in $\overline{\mathcal{F}}$. In fact, any $d$-exact sequence in $\mathcal{F}$:
\begin{align*}
    \xymatrix{
0\ar[r]&c^0\ar[r]& c^1\ar[r]&\cdots\ar[r]& c^d\ar[r]&c^{d+1}\ar[r]&0,
}
\end{align*}
can be decomposed into short exact sequences which correspond to triangles in $\mathcal{D}^b(\mmod \Phi)$. Hence, we obtain a diagram of the form (\ref{diagram_angle2}), and so, by \cite[theorem\ 1]{GKO}, a $(d+2)$-angle in $\overline{\mathcal{F}}$ of the form
\begin{align*}
    \xymatrix{
c^0\ar[r]& c^1\ar[r]&\cdots\ar[r]& c^d\ar[r]&c^{d+1}\ar[r]&\Sigma^d c^0.
}
\end{align*}
\end{remark}

Using the above lemmas and \cite[theorem\ A]{HJV}, we prove there is a bijection between functorially finite wide subcategories of $\mathcal{F}$ as defined in \cite[definition\ 2.11]{HJV}, and functorially finite wide subcategories of $\overline{\mathcal{F}}$ as defined below.

\begin{defn}[{\cite[section\ 1]{HJV}}]\label{defn_d-hom}
Let $\Gamma$ be a finite dimensional $k$-algebra and $\mathcal{G}\subseteq \mmod \Gamma$ be a $d$-cluster tilting subcategory. We say that $(\Gamma, \mathcal{G})$ is a \textit{$d$-homological pair}.

If $\phi:\Phi\rightarrow \Gamma$ is a homomorphism of algebras, then we denote by $\phi_*:\mmod\Gamma\rightarrow \mmod \Phi$ the functor given by restriction of scalars from $\Gamma$ to $\Phi$. Moreover, if $\phi$ is an epimorphism of algebras such that $\phi_*(\mathcal{G})\subseteq \mathcal{F}$ and Tor$^{\Phi}_d (\Gamma,\Gamma)=0$, then we say that $\phi:(\Phi,\mathcal{F})\rightarrow (\Gamma,\mathcal{G})$ is a \textit{$d$-pseudoflat epimorphism of $d$-homological pairs}.
\end{defn}

\begin{remark}\label{rmk_thmA}
In the situation of Definition \ref{defn_d-hom}, we also denote by $\phi_*$ the induced functor on the level of bounded derived categories:
\begin{align*}
    \phi_*:\mathcal{D}^b(\mmod \Gamma)\rightarrow \mathcal{D}^b(\mmod \Phi).
\end{align*}
This is full, faithful and triangulated, since $\phi$ is a homological epimorphism by \cite[proposition\ 5.8]{HJV}. Note that, since $\phi_*$ is triangulated, it commutes with $\Sigma$. Moreover, by Lemma \ref{lemma_dsums}, $\phi_*(\overline{\mathcal{G}})$ is closed under direct summands. Hence
\begin{align*}
\phi_* (\overline{\mathcal{G}})=\phi_* (\add(\Sigma^{\mathbb{Z}d}\mathcal{G}))=\add(\Sigma^{\mathbb{Z}d}(\phi_*(\mathcal{G})))=\overline{\phi_*(\mathcal{G})}.
\end{align*}
\end{remark}

\begin{proof}[Proof of Theorem A]
We start by showing that if $\mathcal{W}\subseteq \mathcal{F}$ is a functorially finite wide subcategory of $\mathcal{F}$, then $\overline{\mathcal{W}}$ is a functorially finite wide subcategory of $\overline{\mathcal{F}}$. By \cite[theorem\ A]{HJV}, there is a $d$-pseudoflat epimorphism of $d$-homological pairs $\phi:(\Phi, \mathcal{F})\rightarrow (\Gamma,\mathcal{G})$ such that $\phi_* (\mathcal{G})=\mathcal{W}$.
Then, by Remark \ref{rmk_thmA}, we have
\begin{align*}
\phi_* (\overline{\mathcal{G}})=\overline{\phi_*(\mathcal{G})}=\overline{\mathcal{W}}\subseteq \overline{\mathcal{F}}.
\end{align*}
Note that $\phi_* (\overline{\mathcal{G}})$ is functorially finite in $\overline{\mathcal{F}}$ by Lemma \ref{lemma_func_fin}.
Then, to complete the first part of the proof, it remains to show that $\phi_* (\overline{\mathcal{G}})$ is closed under $d$-extensions. Take any morphism $\delta:\phi_*(\overline{g})\rightarrow\phi_*(\overline{g'})$ in $\phi_*(\overline{\mathcal{G}})$. Since $\phi_*$ is full and faithful, then $\delta=\phi_*(\overline{g}\xrightarrow{\gamma}\overline{g'})$, for some morphism $\gamma$ in $\overline{\mathcal{G}}$. As $\overline{\mathcal{G}}$ is $(d+2)$-angulated, we can extend $\gamma$ to a $(d+2)$-angle in $\overline{\mathcal{G}}$ of the form:
\begin{align*}
\xymatrix{
\Sigma^{-d}(\overline{g'})\ar[r]& \overline{g^1}\ar[r]&\cdots\ar[r]& \overline{g^d}\ar[r]&\overline{g}\ar[r]^\gamma&\overline{g'}.
}
\end{align*}
Then, by Lemma \ref{lemma_d+2angulated}, we obtain a $(d+2)$-angle in $\overline{\mathcal{F}}$ with objects from $\overline{\mathcal{W}}$:
\begin{align*}
\xymatrix{
\Sigma^{-d}\phi_*(\overline{g'})\ar[r]& \phi_*(\overline{g^1})\ar[r]&\cdots\ar[r]& \phi_*(\overline{g^d})\ar[r]&\phi_*(\overline{g})\ar[r]^{\delta}&\phi_*(\overline{g'}).
}
\end{align*}
Hence $\phi_*(\overline{\mathcal{G}})$ is closed under $d$-extensions.

Now let $\mathcal{X}\subseteq\overline{\mathcal{F}}$ be a functorially finite wide subcategory. Then, by Lemmas \ref{lemma_bij_gen} and \ref{lemma_func_fin}, we have that $\mathcal{X}=\overline{\mathcal{V}}$ for some functorially finite subcategory $\mathcal{V}\subseteq\mathcal{F}$. It remains to show that $\mathcal{V}\subseteq\mathcal{F}$ is wide, in the sense of \cite[definition\ 2.11]{HJV}. Let $\nu: v\rightarrow v'$ be a morphism in $\mathcal{V}$. Since $\mathcal{X}\subseteq\overline{\mathcal{F}}$ is wide, there is a $(d+2)$-angle in $\overline{\mathcal{F}}$ with objects from $\mathcal{X}$ of the form:
\begin{align}\label{eqn_wide}
\xymatrix{
\Sigma^{-d} v'\ar[r] &x^1\ar[r]^{\xi^1}&x^2\ar[r]^{\xi^2}&\cdots\ar[r]^{\xi^{d-1}}& x^d\ar[r]^{\xi^d}&v\ar[r]^{\nu}&v'.
}
\end{align}
Note that $v,\, v'$ are chain complexes concentrated in degree zero since they are in $\mathcal{V}$. Also, as $\mathcal{X}=\overline{\mathcal{V}}$, any $x\in\mathcal{X}$ is isomorphic to a complex with zero differentials and so $H(x)\cong x$. For $i=1,\dots,\,d$, let $v^i$ and $\nu^i$ be the components at degree zero of $x^i$ and $\xi^i$ respectively, and note that $v^i\in\mathcal{V}$.

Note that $H^{0}(-)=\Hom_{\mathcal{D}^b}(\Phi,-)$. Since $\Phi$ is a projective module in $\mmod\Phi$, then $\Phi\in\overline{\mathcal{F}}$. Then, applying  $H^{0}(-)=\Hom_{\overline{\mathcal{F}}}(\Phi,-)$ to (\ref{eqn_wide}), by \cite[proposition\ 2.5]{GKO} we obtain the exact sequence:
\begin{align*}
\xymatrix{
0\ar[r] &v^1\ar[r]^{\nu^1}&v^2\ar[r]^{\nu^2}&\cdots\ar[r]^{\nu^{d-1}}& v^d\ar[r]^{\nu^d}&v\ar[r]^{\nu}&v',
}
\end{align*}
where we have used the fact that $H^0(\Sigma^{-d} v')=0$, since $v'$ is concentrated in degree zero. Then, by Lemma \ref{lemma_d-kernel}, we conclude that
\begin{align*}
\xymatrix{
0\ar[r] &v^1\ar[r]^{\nu^1}&v^2\ar[r]^{\nu^2}&\cdots\ar[r]^{\nu^{d-1}}& v^d\ar[r]^{\nu^d}&v
}
\end{align*}
is a $d$-kernel of $\nu$ in $\mathcal{F}$ with objects from $\mathcal{V}$.

The existence of a $d$-cokernel of $\nu$ in $\mathcal{F}$ with objects from $\mathcal{V}$ follows by a dual argument.

Consider a $d$-exact sequence in $\mathcal{F}$ of the form:
\begin{align*}
    \xymatrix{
0\ar[r] & v'\ar[r]^{\varphi^0}&f^1\ar[r]^{\varphi^1}&\cdots\ar[r]^{\varphi^{d-2}}& f^{d-1}\ar[r]^-{\varphi^{d-1}}&f^d\ar[r]^{\varphi^d}& v\ar[r]&0,
}
\end{align*}
with $v$ and $v'$ in $\mathcal{V}$. Then, by Remark \ref{rmk_angle_seq}, there is a $(d+2)$-angle in $\overline{\mathcal{F}}$ of the form:
\begin{align*}
    \xymatrix{
v'\ar[r]^{\varphi^0}&f^1\ar[r]^{\varphi^1}&f^2\ar[r]^{\varphi^2}&\cdots\ar[r]^{\varphi^{d-2}}& f^{d-1}\ar[r]^-{\varphi^{d-1}}&f^d\ar[r]^{\varphi^d}& v\ar[r]^-{\alpha}&\Sigma^d v'.
}
\end{align*}
Since $\mathcal{X}$ is closed under $d$-extensions and $v,\,v'\in\mathcal{X}$, there is a $(d+2)$-angle in $\overline{\mathcal{F}}$ with objects from $\mathcal{X}$:
\begin{align}\label{diagram_d-exseq}
\xymatrix{
v'\ar[r]^{\xi^0} &x^1\ar[r]^{\xi^1}&x^2\ar[r]^{\xi^2}&\cdots\ar[r]^{\xi^{d-2}}&x^{d-1}\ar[r]^{\xi^{d-1}}& x^d\ar[r]^{\xi^d}&v\ar[r]^-{\alpha}&\Sigma^d v'.
}
\end{align}
For $i=0,\dots,\,d$, let $v^i$ and $\nu^i$ be the components at degree zero of $x^i$ and $\xi^i$ respectively, and note that $v^i\in\mathcal{V}$.
Applying $H^{0}(-)=\Hom_{\overline{\mathcal{F}}}(\Phi,-)$ to (\ref{diagram_d-exseq}), we obtain the exact sequence:
\begin{align*}
\xymatrix{
0\ar[r]& v'\ar[r]^{\nu^0} &v^1\ar[r]^{\nu^1}&v^2\ar[r]^{\nu^2}&\cdots\ar[r]^{\nu^{d-1}}& v^d\ar[r]^{\nu^d}&v\ar[r]&0.
}
\end{align*}
By Lemma \ref{lemma_d-kernel} and its dual, this is a $d$-exact sequence. Moreover, by axiom (N3) from Definition \ref{defn_angles}, we have the morphism of $(d+2)$-angles in $\overline{\mathcal{F}}$:
\begin{align*}
    \xymatrix {
v'\ar[r]^{\varphi^0}\ar@{=}[d] & f^1\ar[r]^{\varphi^1}\ar@{-->}[d]^{\zeta^1}& f^2\ar[r]^{\varphi^2}\ar@{-->}[d]^{\zeta^2} &\cdots\ar[r]^{\varphi^{d-1}} & f^{d}\ar[r]^{\varphi^d}\ar@{-->}[d]^{\zeta^d} &v\ar[r]^-{\alpha}\ar@{=}[d] &\Sigma^d v'\ar@{=}[d]\\
v'\ar[r]_{\xi^0} &x^1\ar[r]_{\xi^1}&x^2\ar[r]_{\xi^2}&\cdots\ar[r]_{\xi^{d-1}}& x^d\ar[r]_{\xi^d}&v\ar[r]_-{\alpha}&\Sigma^d v'.
}
\end{align*}
Applying $H^0(-)$ to the above, we obtain the commutative diagram:
\begin{align*}
    \xymatrix {
0\ar[r]&v'\ar[r]^{\varphi^0}\ar@{=}[d] & f^1\ar[r]^{\varphi^1}\ar[d]^{\psi^1}& f^2\ar[r]^{\varphi^2}\ar[d]^{\psi^2} &\cdots\ar[r]^{\varphi^{d-1}} & f^{d}\ar[r]^{\varphi^d}\ar[d]^{\psi^d} &v\ar[r]\ar@{=}[d] &0\\
0\ar[r]& v'\ar[r]_{\nu^0} &v^1\ar[r]_{\nu^1}&v^2\ar[r]_{\nu^2}&\cdots\ar[r]_{\nu^{d-1}}& v^d\ar[r]_{\nu^d}&v\ar[r]&0.
}
\end{align*}
Hence, the first and second row in the above diagram are two Yoneda equivalent $d$-exact sequences, and the second row has objects in $\mathcal{V}$.
\end{proof}

\section{Auslander-Reiten $(d+2)$-angles}\label{section_AR}
Let us go back to Setup \ref{setup}.
In this section we introduce and study Auslander-Reiten $(d+2)$-angles both in $\mathcal{M}$ and in its subcategory $\mathcal{W}$.

The definition of Auslander-Reiten $(d+2)$-angles was first introduced by Iyama and Yoshino in \cite[definition\ 3.8]{IY}. Here we present a modified definition since we force the end terms of any Auslander-Reiten $(d+2)$-angle to have local endomorphism rings, as pointed out in Remark \ref{remark_AR}. 

\begin{defn}
A $(d+2)$-angle in $\mathcal{M}$ of the form
\begin{align*}
\xymatrix {
\epsilon:& X^0\ar[r]^{\xi^0} & X^1\ar[r]^{\xi^1}& X^2\ar[r] &\cdots\ar[r] & X^d\ar[r]^{\xi^d} & X^{d+1}\ar[r]^{\xi^{d+1}} &\Sigma^d X^0
}
\end{align*}
is an \textit{Auslander-Reiten $(d+2)$-angle} if $\xi^0$ is left almost split, $\xi^d$ is right almost split and, when $d\geq2$, also $\xi^1,\dots,\, \xi^{d-1}\in\rad_\mathcal{M}$.
\end{defn}

\begin{remark}\label{remark_AR}
Suppose $\epsilon$ as in the above definition is an Auslander-Reiten $(d+2)$-angle. Since $\xi^0$ is left almost split, \cite[lemma\ 2.3]{KH} implies that $\End(X^0)$ is local and hence $X^0$ is indecomposable. Similarly, since $\xi^d$ is right almost split, then $\End(X^{d+1})$ is local and hence $X^{d+1}$ is indecomposable.
Also, as proved in the more general Lemma \ref{lemma_indec_subc}, we have that $\xi^0$ and $\xi^d$ are in $\rad_{\mathcal{M}}$.

Moreover, when $d=1$, we have $\xi^0$ and $\xi^d$ in $\rad_{\mathcal{M}}$, so that $\xi^d$ is right minimal and $\xi^0$ is left minimal.
When $d\geq 2$, since $\xi^{d-1}\in \rad_\mathcal{M}$, by Lemma \ref{lemma_lrmin} we have that $\xi^d$ is right minimal and similarly $\xi^0$ is left minimal.
\end{remark}

We now give equivalent definitions for Auslander-Reiten $(d+2)$-angles.

\begin{lemma}\label{lemma_lras}
Let
\begin{align*}
\xymatrix {
\epsilon:& X^0\ar[r]^{\xi^0} & X^1\ar[r]^{\xi^1}& X^2\ar[r] &\cdots\ar[r] & X^d\ar[r]^{\xi^d} & X^{d+1}\ar[r]^{\xi^{d+1}} &\Sigma^d X^0
}
\end{align*}
be a $(d+2)$-angle. Then the following are equivalent:
\begin{enumerate}[label=(\alph*)]
\item $\epsilon$ is an Auslander-Reiten $(d+2)$-angle,
\item $\xi^0,\,\xi^1,\dots,\, \xi^{d-1}$ are in $\rad_\mathcal{M}$ and $\xi^d$ is right almost split,
\item $\xi^1,\dots,\, \xi^{d-1},\,\xi^d$ are in $\rad_\mathcal{M}$ and $\xi^0$ is left almost split.
\end{enumerate}
\end{lemma}

Instead of proving Lemma \ref{lemma_lras} now, we will later prove the more general Lemma \ref{lemma_lras_gen}. The case $\mathcal{W}=\mathcal{M}$ in the latter will then give us Lemma \ref{lemma_lras}.

The following lemma is the generalisation of \cite[lemma\ 2.6]{KH} to $(d+2)$-angles.

\begin{lemma}\label{lemma_AReq}
Consider a $(d+2)$-angle of the form
\begin{align*}
\xymatrix {
\epsilon:& X^0\ar[r]^{\xi^0} & X^1\ar[r]^{\xi^1}& X^2\ar[r] &\cdots\ar[r] & X^d\ar[r]^{\xi^d} & X^{d+1}\ar[r]^{\xi^{d+1}} &\Sigma^d X^0,
}
\end{align*}
and suppose that $\xi^d$ is right almost split and, if $d\geq 2$, also that $\xi^1,\dots,\,\xi^{d-1}$ are in $\rad_\mathcal{M}$. Then the following are equivalent:
\begin{enumerate}[label=(\alph*)]
\item $\End(X^0)$ is local,
\item $\xi^{d+1}$ is left minimal,
\item $\xi^0$ is in $\rad_{\mathcal{M}}$,
\item $\epsilon$ is an Auslander-Reiten $(d+2)$-angle.
\end{enumerate}
\end{lemma}

 As done for Lemma \ref{lemma_lras}, instead of proving the above now, we will later prove the more general Lemma \ref{lemma_AReq_W}. The case $\mathcal{W}=\mathcal{M}$ in the latter will then give us Lemma \ref{lemma_AReq}.

We now define Auslander-Reiten $(d+2)$-angles in the additive subcategory $\mathcal{W}$ of $\mathcal{M}$ closed under  $d$-extensions. To do so, we first define left and right almost split morphisms in $\mathcal{W}$.

\begin{lemma}\label{lemma_indec_subc}
\begin{enumerate}[label=(\alph*)]
\item Let $\omega^0: W^0\rightarrow W^1$ be left almost split in $\mathcal{W}$, then $\End (W^0)$ is local and $\omega^0\in\rad_\mathcal{W}$.
\item Let $\omega^d:W^d\rightarrow W^{d+1}$ be right almost split in $\mathcal{W}$, then $\End(W^{d+1})$ is local and $\omega^d\in\rad_\mathcal{W}$.
\end{enumerate}
\end{lemma}

\begin{proof}
We only prove (a), the proof for (b) is then dual. Suppose $\omega^0: W^0\rightarrow W^1$ is left almost split in $\mathcal{W}$. Let $\mu,\, \nu: W^0\rightarrow W^0$ be morphisms that are not split monomorphisms. Then there are morphisms $\mu',\, \nu':W^1\rightarrow W^0$ such that $\mu=\mu'\circ\omega^0$ and $\nu=\nu'\circ\omega^0$. By \cite[proposition\ 15.15]{AF}, in order to prove that $\End(W^0)$ is local, it is enough to prove that $\mu+\nu$ is not a split monomorphism.

Suppose for a contradiction that $\mu+\nu$ is a split monomorphism.
Hence there is a morphism $\gamma:W^0\rightarrow W^0$ such that $\gamma\circ (\mu+\nu)=1_{W^0}$. Then
\begin{align*}
    \gamma\circ(\mu'+\nu')\circ\omega^0=\gamma\circ (\mu+\nu)=1_{W^0}.
\end{align*}
Hence $\omega^0$ is a split monomorphism, contradicting our initial assumption. So $\End(W^0)$ is local.
%Extend $\omega^0$ and $\mu+\nu$ to $(d+2)$-angles. Then, since $(\mu'+\nu')\circ \omega^0=\mu+\nu$, by axiom (N3) from Definition \ref{defn_angles}, we can construct a commutative diagram of the form:
%\begin{align*}
%    \xymatrix {
%W^0\ar[r]^{\omega^0}\ar@{=}[d] & W^1\ar[r]^{\xi^1}\ar[d]^{\mu'+\nu'}& X^2\ar[r]\ar@{-->}[d]^{\psi^2} &\cdots\ar[r] & X^d\ar[r]^{\xi^d}\ar@{-->}[d]^{\psi^d} & X^{d+1}\ar[r]^{\xi^{d+1}}\ar@{-->}[d]^{\psi^{d+1}} &\Sigma^d W^0\ar@{=}[d]\\
%W^0\ar[r]_{\mu+\nu} & W^0\ar[r]_{\eta^1}& Y^2\ar[r] &\cdots\ar[r] & Y^d\ar[r]_{\eta^d} & %Y^{d+1}\ar[r]_{\eta^{d+1}} &\Sigma^d W^0.
%}
%\end{align*}
%Then, since $\mu+\nu$ is a split monomorphism, by Lemma \ref{lemma_epimono} it follows that $\eta^{d+1}=0$. Hence
%\begin{align*}
%    \xi^{d+1}=\eta^{d+1}\circ\psi^{d+1}=0,
%\end{align*}
%and $\omega^0$ is a split monomorphism by Lemma \ref{lemma_epimono}, contradicting our initial assumption. Hence $\End(W^0)$ is local.

Since $\End(W^0)$ is local, it follows that $W^0$ is indecomposable. Since $\mathcal{M}$ is Krull-Schmidt, there are indecomposable objects $W_1,\dots,\,W_t$ such that
\begin{align*}
W^1=W_1\oplus\cdots\oplus W_t.
\end{align*}
Moreover, $W_1,\dots,\,W_t$ are in $\mathcal{W}$ since $\mathcal{W}$ is closed under summands. Then we have
\begin{align*}
\omega^0=
\begin{psmallmatrix}
\omega^0_1\\ \vdots \\ \omega^0_t
\end{psmallmatrix}
:W^0\rightarrow W_1\oplus\cdots\oplus W_t.
\end{align*}
Suppose there is some $i\in\{ 1,\dots,\,t \}$ such that $\omega^0_i: W^0\rightarrow W_i$ is not in $\rad_\mathcal{W}$. Then, since $W^0$ and $W_i$ are both indecomposable, it follows that $\omega^0_i$ is invertible. Hence
\begin{align*}
\begin{psmallmatrix}
0& \cdots &0 & (\omega^0_i)^{-1} &0&\cdots &0
\end{psmallmatrix}
\circ \omega^0=
(\omega^0_i)^{-1}\circ \omega^0_i=
1_{W^0},
\end{align*} 
contradicting the fact that $\omega^0$ is not a split monomorphism. Hence such an $i$ does not exist and $\omega^0\in\rad_\mathcal{W}$.
\end{proof}

\begin{defn}\label{defn_AR_subc}
A $(d+2)$-angle of the form
\begin{align*}
\xymatrix {
\epsilon:& W^0\ar[r]^{\omega^0} & W^1\ar[r]^{\omega^1}& W^2\ar[r] &\cdots\ar[r] & W^d\ar[r]^{\omega^d} & W^{d+1}\ar[r]^{\omega^{d+1}} &\Sigma^d W^0,
}
\end{align*}
with $W^0,\,W^1,\dots,\, W^{d+1}$ in $\mathcal{W}$ is an \textit{Auslander-Reiten $(d+2)$-angle in $\mathcal{W}$} if the morphism $\omega^0$ is left almost split in $\mathcal{W}$, the morphism $\omega^d$ is right almost split in $\mathcal{W}$ and, when $d\geq 2$, also  $\omega^1,\dots,\,\omega^{d-1}$ are in $\rad_\mathcal{W}$.
\end{defn}

\begin{remark}\label{remark_rad}
Note that since $\mathcal{W}$ is a full subcategory of $\mathcal{M}$, then $\rad_\mathcal{W}$ is equal to the restriction of $\rad_\mathcal{M}$ to $\mathcal{W}$.
\end{remark}

\begin{lemma}\label{lemma_lm_lasW}
Let
\begin{align*}
\xymatrix {
\epsilon:& W^0\ar[r]^{\omega^0} & W^1\ar[r]^{\omega^1}& W^2\ar[r] &\cdots\ar[r] & W^d\ar[r]^{\omega^d} & W^{d+1}\ar[r]^{\omega^{d+1}} &\Sigma^d W^0
}
\end{align*}
be a $(d+2)$-angle with $W^0,\,W^1,\dots,\,W^{d+1}$ in $\mathcal{W}$. If $\omega^d$ is right almost split in $\mathcal{W}$ and $\omega^{d+1}$ is left minimal, then $\omega^0$ is left almost split in $\mathcal{W}$.
\end{lemma}

\begin{proof}
Since $\omega^d$ is not a split epimorphism, Lemma \ref{lemma_epimono} implies that $\omega^0$ is not a split monomorphism. Let $\phi^0: W^0\rightarrow V^0$ be a morphism in $\mathcal{W}$ that is not a split monomorphism.
Extend $\Sigma^d (\phi^0)\circ \omega^{d+1}$ to a $(d+2)$-angle and consider the following commutative diagram, built using axiom (N3) from Definition \ref{defn_angles}:
\begin{align*}
    \xymatrix {
 W^0\ar[r]^{\omega^0}\ar[d]^{\phi^0}& W^1\ar[r]^{\omega^1}\ar@{-->}[d]^{\phi^1}& W^2\ar[r]\ar@{-->}[d]^{\phi^2} &\cdots\ar[r]  & W^d\ar[r]^{\omega^{d}}\ar@{-->}[d]^{\phi^{d}} &W^{d+1}\ar[rrr]^{\omega^{d+1}}\ar@{=}[d]&&& \Sigma^d W^0\ar[d]^{\Sigma^d (\phi^0)}\\
V^0\ar[r]_{\eta^0}& V^1\ar[r]_{\eta^1} & V^2\ar[r] &\cdots\ar[r] & V^d\ar[r]_{\eta^d} &W^{d+1}\ar[rrr]_{\Sigma^d (\phi^0)\circ \omega^{d+1}} &&&\Sigma^d V^0,
}
\end{align*}
where, as $V^0$ and $W^{d+1}$ are in $\mathcal{W}$, which is closed under $d$-extensions, we can choose $V^1,\dots,\, V^d$ in $\mathcal{W}$.

Suppose for a contradiction that $\eta^0$ is not a split monomorphism. Then $\eta^d$ is not a split epimorphism by Lemma \ref{lemma_epimono}. As $\omega^d$ is right almost split in $\mathcal{W}$ and $\eta^d:V^d\rightarrow W^{d+1}$ is a morphism in $\mathcal{W}$, then there is a morphism $\psi^d:V^d\rightarrow W^d$ such that $\eta^d=\omega^d\circ\psi^d$. So we can construct a commutative diagram of the form:
\begin{align*}
    \xymatrix {
V^0\ar[r]^{\eta^0}\ar@{-->}[d]^{\psi^0}& V^1\ar[r]\ar@{-->}[d]^{\psi^1}&\cdots\ar[r] & V^{d-1}\ar[r]^{\eta^{d-1}}\ar@{-->}[d]^{\psi^{d-1}}& V^{d}\ar[r]^{\eta^{d}}\ar[d]^{\psi^{d}} & W^{d+1}\ar[rrr]^{\Sigma^d (\phi^0)\circ\omega^{d+1}}\ar@{=}[d] &&& \Sigma^d V^{0}\ar@{-->}[d]^{\Sigma^d (\psi^0)}\\
W^0\ar[r]_{\omega^0}& W^1\ar[r] &\cdots\ar[r]& W^{d-1}\ar[r]_{\omega^{d-1}} & W^{d}\ar[r]_{\omega^{d}} & W^{d+1}\ar[rrr]_{\omega^{d+1}} &&& \Sigma^d W^{0}.
}
\end{align*}
Hence we have
\begin{align*}
    \Sigma^{d}(\psi^0\phi^0)\circ\omega^{d+1}=
    \Sigma^{d}(\psi^0)\circ \Sigma^d (\phi^0)\circ\omega^{d+1}= \omega^{d+1}.
\end{align*}
Since $\omega^{d+1}$ is left minimal, then $\Sigma^{d}(\psi^0\phi^0)$ is an isomorphism, and so also $\psi^0\phi^0$ is an isomorphism, contradicting our assumption that $\phi^0$ is not a split monomorphism. Hence $\eta^0$ is a split monomorphism and there is a morphism $\gamma:V^1\rightarrow V^0$ such that $\gamma\circ\eta^0=1_{V^0}$. Then
\begin{align*}
    \gamma\phi^1\omega^0 = \gamma\eta^0\phi^0=1_{V^0}\circ \phi^0=\phi^0,
\end{align*}
and so $\omega^0$ is left almost split in $\mathcal{W}$.
\end{proof}

\begin{lemma}\label{lemma_lras_gen}
Let
\begin{align*}
\xymatrix {
\epsilon:& W^0\ar[r]^{\omega^0} & W^1\ar[r]^{\omega^1}& W^2\ar[r] &\cdots\ar[r] & W^d\ar[r]^{\omega^d} & W^{d+1}\ar[r]^{\omega^{d+1}} &\Sigma^d W^0
}
\end{align*}
be a $(d+2)$-angle with $W^0,\,W^1,\dots,\,W^{d+1}$ in $\mathcal{W}$. Then the following are equivalent:
\begin{enumerate}[label=(\alph*)]
\item $\epsilon$ is an Auslander-Reiten $(d+2)$-angle in $\mathcal{W}$,
\item $\omega^0,\,\omega^1,\dots,\, \omega^{d-1}$ are in $\rad_\mathcal{W}$ and $\omega^d$ is right almost split in $\mathcal{W}$.
\item $\omega^1,\dots,\, \omega^{d-1},\,\omega^d$ are in $\rad_\mathcal{W}$ and $\omega^0$ is left almost split in $\mathcal{W}$.
\end{enumerate}
\end{lemma}

\begin{proof}
Note that (a) implies both (b) and (c) by Lemma \ref{lemma_indec_subc} and Definition \ref{defn_AR_subc}. 
Suppose now that (b) holds. Since $\omega^0$ is in $\rad_\mathcal{W}$ and so in $\rad_\mathcal{M}$, then so is $(-1)^d\Sigma^d(\omega^0)$ and $\omega^{d+1}$ is left minimal by Lemma \ref{lemma_lrmin}. Then, by Lemma \ref{lemma_lm_lasW}, it follows that $\omega^0$ is left almost split in $\mathcal{W}$, so (c) holds as $\omega^d\in\rad_\mathcal{M}$ by Lemma \ref{lemma_indec_subc}.

The fact that (c) implies (b) follows by a dual argument and so it is now clear that they both imply (a).
\end{proof}

\begin{lemma}\label{lemma_AReq_W}
Consider a $(d+2)$-angle of the form
\begin{align*}
\xymatrix {
\epsilon:& W^0\ar[r]^{\omega^0} & W^1\ar[r]^{\omega^1}& W^2\ar[r] &\cdots\ar[r] & W^d\ar[r]^{\omega^d} & W^{d+1}\ar[r]^{\omega^{d+1}} &\Sigma^d W^0,
}
\end{align*}
with $W^0,\,W^1,\dots,\,W^{d+1}$ in $\mathcal{W}$ and suppose that $\omega^d$ is right almost split in $\mathcal{W}$ and, if $d\geq 2$, also that $\omega^1,\dots,\,\omega^{d-1}$ are in $\rad_\mathcal{W}$. Then the following are equivalent:
\begin{enumerate}[label=(\alph*)]
\item $\End(W^0)$ is local,
\item $\omega^{d+1}$ is left minimal,
\item $\omega^0$ is in $\rad_\mathcal{W}$,
\item $\epsilon$ is an Auslander-Reiten $(d+2)$-angle in $\mathcal{W}$.
\end{enumerate}
\end{lemma}

\begin{proof}
(a)$\Rightarrow$(b). Suppose $\End(W^0)$ is local. By Lemma \ref{lemma_epimono}, since $\omega^d$ is not a split epimorphism, it follows that $\omega^{d+1}$ is non-zero. Then, as $\End(W^0)\cong \End(\Sigma^d W^0)$ is local, it follows that $\omega^{d+1}$ is left minimal by \cite[lemma\ 2.4]{KH}.

(b)$\Rightarrow$(d). Suppose $\omega^{d+1}$ is left minimal, then Lemma \ref{lemma_lm_lasW} implies that $\omega^0$ is left almost split in $\mathcal{W}$.

(d)$\Rightarrow$(a). Suppose $\epsilon$ is an Auslander-Reiten $(d+2)$-angle in $\mathcal{W}$. Then $\omega^0$ is left almost split in $\mathcal{W}$ and by lemma \ref{lemma_indec_subc}, we have that $\End(W^0)$ is local.

(c)$\Rightarrow$(d). Suppose $\omega^0$ is in $\rad_\mathcal{W}$. Then, by Lemma \ref{lemma_lras_gen}, it follows that $\epsilon$ is an Auslander-Reiten $(d+2)$-angle in $\mathcal{W}$.

(b)$\Rightarrow$(c). Suppose $\omega^{d+1}$ is left minimal. Lemma \ref{lemma_lrmin} implies that $(-1)^d\Sigma^d(\omega^0)\in\rad_\mathcal{M}$, so $\omega^0\in\rad_\mathcal{M}$ and $\omega^0\in\rad_\mathcal{W}$ by Remark \ref{remark_rad}.
\end{proof}

\section{Proof of Theorem B}\label{section_B}
In this section, we generalise \cite[theorem\ 3.1]{JP} to any $d\geq 1$, see Theorem B. To do so, we start by proving the higher version of \cite[lemmas\ 2.2 and 2.3]{JP} and another lemma.

We work in Setup \ref{setup}.

\begin{lemma}\label{lemma_simpsoc}
Consider an Auslander-Reiten $(d+2)$-angle in $\mathcal{M}$ of the form
\begin{align*}
\xymatrix {
\epsilon:& X^0\ar[r]^{\xi^0} & X^1\ar[r]^{\xi^1}& X^2\ar[r] &\cdots\ar[r] & X^d\ar[r]^-{\xi^d} & X^{d+1}\ar[r]^{\xi^{d+1}} &\Sigma^d X^0.
}
\end{align*}
View the abelian group $\Hom(X^{d+1}, \Sigma^d X^0)$ as an $\End(X^{d+1})$-right-module via composition of morphisms. The socle of this module is simple and equal to the submodule generated by $\xi^{d+1}$.
\end{lemma}

\begin{proof}
Let $M$ be a non-zero submodule of $\Hom(X^{d+1},\Sigma^d X^0)$ and pick a non-zero element $\mu: X^{d+1}\rightarrow \Sigma^d X^0$ in $M$. Extend $\mu$ to a $(d+2)$-angle:
\begin{align*}
\xymatrix {
X^0\ar[r]^{\eta^0} & Y^1\ar[r]^{\eta^1}& Y^2\ar[r] &\cdots\ar[r] & Y^d\ar[r]^{\eta^d} & X^{d+1}\ar[r]^{\mu} &\Sigma^d X^0.
}
\end{align*}
Since $\mu$ is non-zero, then $\eta^0$ is not a split monomorphism by Lemma \ref{lemma_epimono}. Then, as $\xi^0$ is left almost split, there is a morphism $\psi^1: X^1\rightarrow Y^1$ such that $\psi^1\circ\xi^0=\eta^0$. So, by axiom (N3) from Definition \ref{defn_angles}, there exist morphisms $\psi^2,\dots,\,\psi^{d+1}$ making the following diagram commutative:
\begin{align*}
    \xymatrix {
X^0\ar[r]^{\xi^0}\ar@{=}[d] & X^1\ar[r]^{\xi^1}\ar[d]^{\psi^1}& X^2\ar[r]\ar@{-->}[d]^{\psi^2} &\cdots\ar[r] & X^d\ar[r]^{\xi^d}\ar@{-->}[d]^{\psi^d} & X^{d+1}\ar[r]^{\xi^{d+1}}\ar@{-->}[d]^{\psi^{d+1}} &\Sigma^d X^0\ar@{=}[d]\\
X^0\ar[r]_{\eta^0} & Y^1\ar[r]_{\eta^1}& Y^2\ar[r] &\cdots\ar[r] & Y^d\ar[r]_{\eta^d} & X^{d+1}\ar[r]_{\mu} &\Sigma^d X^0.
}
\end{align*}
In particular, we have $\mu\circ\psi^{d+1}=\xi^{d+1}$. So, in the $\End(X^{d+1})$-module $\Hom(X^{d+1},\Sigma^d X^0)$, the element $\xi^{d+1}$ is a multiple of $\mu$. So $\xi^{d+1}$ is in $M$ and the non-zero submodule of $\Hom(X^{d+1},\Sigma^d X^0)$ generated by $\xi^{d+1}$ is contained in $M$. Then, the socle of $\Hom(X^{d+1},\Sigma^d X^0)$ is the submodule generated by $\xi^{d+1}$.

Note that $\End(X^{d+1})$ is local, as $\epsilon$ is an Auslander-Reiten $(d+2)$-angle. Since the socle of $\Hom(X^{d+1},\Sigma^d X^0)$ is generated by the single element $\xi^{d+1}$, it follows that it is simple if it is annihilated by the Jacobson radical of $\End(X^{d+1})$. Let $\rho: X^{d+1}\rightarrow X^{d+1}$ be in the radical of $\End(X^{d+1})$, then by the dual of \cite[proposition\ 15.15(e)]{AF}, we have that $\rho$ has no right inverse. Hence $\rho$ is not a split epimorphism and, since $\xi^d$ is right almost split, there is a morphism $\rho':X^{d+1}\rightarrow X^d$ such that $\rho=\xi^d\circ \rho'$. Then, by Lemma \ref{lemma_consecutive} we have
\begin{align*}
    \xi^{d+1}\circ \rho=\xi^{d+1}\circ\xi^d\circ \rho'=0\circ \rho'=0,
\end{align*}
as we wished to prove.
\end{proof}

\begin{defn} [{\cite[section\ 0]{J}}]
For an additive subcategory $\mathcal{U}\subseteq \mathcal{M}$, we define
\begin{align*}
    \mathcal{U}\text{-exact}=
    \Bigg\{
    \begin{matrix}M^1\rightarrow \cdots \rightarrow M^d \\ \text{ is a complex in } \mathcal{M}\end{matrix}
    \,\, \Bigg\rvert \,\,\begin{matrix}0\rightarrow\Hom_{\mathcal{M}}(U,M^1)\rightarrow \cdots \rightarrow \Hom_{\mathcal{M}}(U,M^d)\rightarrow 0 \\ \text{ is exact for each } U\in\mathcal{U}\end{matrix}
    \Bigg\}.
    \end{align*}
\end{defn}

\begin{lemma}\label{lemma_cover}
Let $W$ be in $\mathcal{W}$ and let
\begin{align*}
\xymatrix {
\epsilon:& X^0\ar[r]^{\xi^0} & X^1\ar[r]^{\xi^1}& X^2\ar[r] &\cdots\ar[r] & X^d\ar[r]^{\xi^d} & W\ar[r]^{\xi^{d+1}} &\Sigma^d X^0
}
\end{align*}
be an Auslander-Reiten $(d+2)$-angle. Suppose $\nu:V\rightarrow X^0$ is a $\mathcal{W}$-cover. Then $V$ is either zero or indecomposable.
\end{lemma}

\begin{proof}
Suppose $V$ is non-zero and recall that $\mathcal{M}$ is Krull-Schmidt. Let $V_i$ be an indecomposable direct summand of $V$, let $\iota_i :V_i\rightarrow V$ be the inclusion of $V_i$ into $V$ and $\nu_i :=\nu\circ\iota_i$.
Extend $\Sigma^d (\nu_i)$ to a $(d+2)$-angle:
\begin{align*}
\xymatrix {
X^0\ar[r]^{\eta^0} & Y^1\ar[r]^{\eta^1}& Y^2\ar[r] &\cdots\ar[r] & Y^d\ar[r]^{\eta^d} & \Sigma^d V_i\ar[r]^-{\Sigma^d (\nu_i)} &\Sigma^d X^0.
}
\end{align*}
Since $\nu$ is a $\mathcal{W}$-cover, then $\nu_i$ is non-zero and so $\Sigma^d (\nu_i)$ is non-zero. Hence $\eta^0$ is not a split monomorphism by Lemma \ref{lemma_epimono} and, as $\xi^0$ is left almost split, there exists a morphism $\psi^1 :X^1\rightarrow Y^1$ such that $\psi^1\circ\xi^0=\eta^0$. Then, by axiom (N3) from Definition \ref{defn_angles}, there are morphisms $\psi^2,\dots,\,\psi^{d+1}$ making the following diagram commutative:
\begin{align*}
    \xymatrix {
X^0\ar[r]^{\xi^0}\ar@{=}[d] & X^1\ar[r]^{\xi^1}\ar[d]^{\psi^1}& X^2\ar[r]\ar@{-->}[d]^{\psi^2} &\cdots\ar[r] & X^d\ar[r]^{\xi^d}\ar@{-->}[d]^{\psi^d} & W\ar[r]^{\xi^{d+1}}\ar@{-->}[d]^{\psi^{d+1}} &\Sigma^d X^0\ar@{=}[d]\\
X^0\ar[r]_{\eta^0} & Y^1\ar[r]_{\eta^1}& Y^2\ar[r] &\cdots\ar[r] & Y^d\ar[r]_{\eta^d} & \Sigma^d V_i\ar[r]_-{\Sigma^d (\nu_i)} &\Sigma^d X^0.
}
\end{align*}
In particular, we have $\Sigma^d (\nu_i)\circ \psi^{d+1}=\xi^{d+1}$. Then, letting $\varphi:=\Sigma^{-d}(\psi^{d+1}):\Sigma^{-d}W\rightarrow V_i$, we have $\nu_i\circ \varphi=\Sigma^{-d}(\xi^{d+1})$. As $\xi^{d+1}$ is non-zero, it follows that $\varphi$ is non-zero.

Hence every indecomposable direct summand $V_i$ of $V$ permits a non-zero morphism $\Sigma^{-d} W\rightarrow V_i$. We complete the proof by showing that at most one indecomposable direct summand of $V$ can permit such a morphism.

Extend $\nu$ to a $(d+2)$-angle of the form
\begin{align*}
\xymatrix {
V\ar[r]^{\nu} & X^0\ar[r]^{\omega^0}& Z^1\ar[r]^{\omega^1}& Z^2\ar[r] &\cdots\ar[r] & Z^{d-1}\ar[r]^-{\omega^{d-1}} & Z^d\ar[r]^{\omega^d} &\Sigma^d V.
}
\end{align*}
Consider the exact sequence
\begin{align*}
    \Hom (W, Z^{d-1})\xrightarrow{\widetilde{\omega^{d-1}}}\Hom(W,Z^d)\xrightarrow{\widetilde{\omega^d}}\Hom(W,\Sigma^d V)\xrightarrow{\phi}\Hom(W,\Sigma^d X^0),
\end{align*}
where, for a morphism $\eta$ we use the notation $\widetilde{\eta}:=\Hom(W,\eta)$  and $\phi:=\widetilde{(-1)^d\Sigma^d(\nu)}$ for readability.
Note that $Z^1\rightarrow \dots \rightarrow Z^d$ is in $\mathcal{W}$-exact by \cite[lemma\ 2.1]{J}. Hence $\widetilde{\omega^{d-1}}$ is surjective, so that $\widetilde{\omega^d}$ is the zero map and $\phi$ is injective. Viewing $\phi$ as a homomorphism of finite-dimensional right modules over the finite-dimensional $k$-algebra $\End(W)$, the target $\Hom(W,\Sigma^d X^0)$ has simple socle by Lemma \ref{lemma_simpsoc}. Hence the image is either zero or indecomposable. Since $\phi$ is injective, then the same is true for the source $\Hom(W,\Sigma^d V)$. So, if $V=V_1\oplus\dots\oplus V_t$, there can be at most one $i\in\{1,\dots,t\}$ such that $\Hom(W,\Sigma^d V_i)\cong \Hom(\Sigma^{-d}W, V_i)$ is non-zero.

Hence, as we claimed, there is at most one indecomposable summand $V_i$ of $V$ permitting a non-zero morphism $\Sigma^{-d}W\rightarrow V_i$.
\end{proof}

\begin{lemma}\label{lemma_sublemma}
Consider an Auslander-Reiten $(d+2)$-angle in $\mathcal{M}$ of the form
\begin{align*}
\xymatrix {
Y^0\ar[r]^{\eta^0} & Y^1\ar[r]^{\eta^1}& Y^2\ar[r] &\cdots\ar[r] & Y^d\ar[r]^{\eta^d} & W\ar[r]^-{\eta^{d+1}} &\Sigma^d Y^0,
}
\end{align*}
and any $U^0\in\mathcal{M}$. Then:
\begin{enumerate}[label=(\alph*)]
    \item for every non-zero morphism $\delta:\Sigma^d U^0\rightarrow\Sigma^d Y^0$, there is a morphism $\phi:W\rightarrow\Sigma^d U^0$ such that $\delta\circ\phi=\eta^{d+1}$;
    \item for every non-zero morphism $\phi:W\rightarrow\Sigma^d U^0$, there is a morphism $\delta:\Sigma^d U^0\rightarrow\Sigma^d Y^0$ such that $\delta\circ\phi=\eta^{d+1}$;
\end{enumerate}
\end{lemma}

\begin{proof}
\begin{enumerate}[label=(\alph*)]
    \item Extend $\delta$ to a $(d+2)$-angle of the form
    \begin{align*}
\xymatrix {
Y^0\ar[r]^{\delta^0}& M^1\ar[r]^{\delta^1}& M^2\ar[r] &\cdots\ar[r] & M^d\ar[r]^{\delta^d} & \Sigma^d U^0\ar[r]^{\delta} &\Sigma^d Y^0.
}
\end{align*}
Since $\delta$ is non-zero, then $\delta^0$ is not a split monomorphism, so there is $\phi^1:Y^1\rightarrow M^1$ such that $\delta^0=\phi^1\circ\eta^0$. So, by axiom (N3) from Definition \ref{defn_angles}, there exist morphisms $\phi^2,\dots,\,\phi^{d+1}$ making the following diagram commutative:
\begin{align*}
    \xymatrix {
Y^0\ar[r]^{\eta^0}\ar@{=}[d] & Y^1\ar[r]^{\eta^1}\ar[d]^{\phi^1}& Y^2\ar[r]\ar@{-->}[d]^{\phi^2} &\cdots\ar[r] & Y^d\ar[r]^{\eta^d}\ar@{-->}[d]^{\phi^d} & W\ar[r]^{\eta^{d+1}}\ar@{-->}[d]^{\phi^{d+1}} &\Sigma^d Y^0\ar@{=}[d]\\
Y^0\ar[r]_{\delta^0} & M^1\ar[r]_{\delta^1}& M^2\ar[r] &\cdots\ar[r] & M^d\ar[r]_{\delta^d} & \Sigma^d U^0\ar[r]_{\delta} &\Sigma^d Y^0.
}
\end{align*}
Then $\phi:=\phi^{d+1}$ is such that $\delta\circ\phi=\eta^{d+1}$.
\item Follows by a dual argument.
\end{enumerate}
\end{proof}

%\begin{theorem}\label{thm_general}
%Let $W$ be in $\mathcal{W}$ and suppose that there exists $U^0$ in $\mathcal{W}$ and a non-zero morphism $\gamma^{d+1}:W\rightarrow \Sigma^d U^0$. Let
%\begin{align*}
%\xymatrix {
%\epsilon:& X^0\ar[r]^{\xi^0} & X^1\ar[r]^{\xi^1}& X^2\ar[r] &\cdots\ar[r] & X^d\ar[r]^{\xi^d} & W\ar[r]^{\xi^{d+1}} &\Sigma^d X^0
%}
%\end{align*}
%be an Auslander-Reiten $(d+2)$-angle in $\mathcal{M}$. Then the following are equivalent:
%\begin{enumerate}[label=(\alph*)]
 %   \item $X^0$ has a $\mathcal{W}$-cover of the form $\varphi:W^0\rightarrow X^0$,
  %  \item there is an Auslander-Reiten $(d+2)$-angle in $\mathcal{W}$ of the form
   % \begin{align*}
%\xymatrix {
%\epsilon':& W^0\ar[r]^{\omega^0} & W^1\ar[r]^{\omega^1}& W^2\ar[r] &\cdots\ar[r] & W^d\ar[r]^{\omega^d} & W\ar[r]^-{\omega^{d+1}} &\Sigma^d W^0.
%}
%\end{align*}
%\end{enumerate}
%\end{theorem}
\begin{proof}[Proof of Theorem B]
We first prove that (a) implies (b). Suppose $\varphi:W^0\rightarrow X^0$ is a $\mathcal{W}$-cover. Extend the non-zero morphism $\gamma^{d+1}$ to a $(d+2)$-angle:
\begin{align*}
\xymatrix {
U^0\ar[r]^{\gamma^0} & U^1\ar[r]^{\gamma^1}& U^2\ar[r] &\cdots\ar[r] & U^d\ar[r]^{\gamma^d} & W \ar[r]^{\gamma^{d+1}} &\Sigma^d U^0,
}
\end{align*}
where we can choose $U^1,\dots,\, U^d$ in $\mathcal{W}$. Note that $\gamma^d$ is not a split epimorphism by Lemma \ref{lemma_epimono}. Since $\xi^d$ is right almost split, there is a morphism $\psi^d:U^d\rightarrow X^d$ such that $\gamma^d=\xi^d\circ\psi^d$. Then there exist morphisms $\psi^0,\dots,\,\psi^{d-1}$ making the following diagram commutative:
\begin{align*}
    \xymatrix {
U^0\ar[r]^{\gamma^0} \ar@{-->}[d]^{\psi^0} & U^1\ar[r]^{\gamma^1} \ar@{-->}[d]^{\psi^1} & \cdots\ar[r] & U^{d-1}\ar[r]^{\gamma^{d-1}} \ar@{-->}[d]^{\psi^{d-1}} & U^d\ar[r]^{\gamma^d} \ar[d]^{\psi^d}  & W\ar[r]^{\gamma^{d+1}} \ar@{=}[d] & \Sigma^d U^0 \ar@{-->}[d]^{\Sigma^d(\psi^0)}\\
X^0\ar[r]_{\xi^0} & X^1\ar[r]_{\xi^1}& \cdots\ar[r] & X^{d-1}\ar[r]_{\xi^{d-1}} & X^d\ar[r]_{\xi^d} & W\ar[r]_{\xi^{d+1}} &\Sigma^d X^0.
}
\end{align*}
In particular, we have $\Sigma^{d}(\psi^0)\circ \gamma^{d+1}=\xi^{d+1}$. Since $\varphi: W^0\rightarrow X^0$ is a $\mathcal{W}$-cover, there is a morphism $\nu :U^0\rightarrow W^0$ such that $ \varphi\circ\nu=\psi^0$. Consider a $(d+2)$-angle extending $\Sigma^d(\nu)\circ \gamma^{d+1}$:
\begin{align*}
\xymatrix {
\epsilon':& W^0\ar[r]^{\omega^0} & W^1\ar[r]^{\omega^1}& W^2\ar[r] &\cdots\ar[r]^{\omega^{d-1}} & W^d\ar[r]^{\omega^d} & W\ar[rr]^{\Sigma^d(\nu)\gamma^{d+1}} &&\Sigma^d W^0,
}
\end{align*}
where, as $W,\,W^0\in\mathcal{W}$, we can choose $W^1,\dots,\, W^d$ in $\mathcal{W}$ and by Lemma \ref{lemma_uni_min}. When $d\geq 2$, we can also choose $\omega^1,\dots,\,\omega^{d-1}$ in $\rad_\mathcal{M}$ and so in $\rad_\mathcal{W}$. We will show that $\epsilon'$ is an Auslander-Reiten $(d+2)$-angle in $\mathcal{W}$.

By Lemma \ref{lemma_cover}, we have that $W^0$ is either zero or indecomposable. Since
\begin{align}\label{eqn_xi}
    0\neq \xi^{d+1}=\Sigma^d(\psi^0)\circ\gamma^{d+1}=\Sigma^d (\varphi\circ\nu)\circ \gamma^{d+1}=\Sigma^d (\varphi)\circ \Sigma^d(\nu)\circ \gamma^{d+1},
\end{align}
it follows that $\Sigma^d(\nu)\circ\gamma^{d+1}$ is non-zero. Then $\Sigma^d W^0$ is non-zero and so $W^0$ is non-zero, hence it is indecomposable, so $\End(W^0)$ is local.

Now, by Lemma \ref{lemma_AReq_W}, in order to prove that $\epsilon'$ is an Auslander-Reiten $(d+2)$-angle in $\mathcal{W}$, it is enough to prove that $\omega^d$ is right almost split in $\mathcal{W}$.

Extend $\varphi:W^0\rightarrow X^0$ to a $(d+2)$-angle:
\begin{align*}
    \xymatrix{
    W^0\ar[r]^{\varphi} & X^0\ar[r]^{\delta^0} & Y^1\ar[r] &\cdots\ar[r]^-{\delta^{d-1}} & Y^d\ar[r]^-{\delta^d} & \Sigma^d W^0.
    }
\end{align*}

Since $\Sigma^d(\nu)\circ \gamma^{d+1}$ is non-zero, Lemma \ref{lemma_epimono} implies that $\omega^d$ is not a split epimorphism.
By (\ref{eqn_xi}), we have that $\varphi\circ\nu\circ\Sigma^{-d}(\gamma^{d+1})=\Sigma^{-d}(\xi^{d+1})$ and so there are morphisms $\alpha^1,\dots,\,\alpha^d$ making the following diagram commutative:
\begin{align}\label{diagram_atob}
    \xymatrix {
\Sigma^{-d}W\ar[rrr]^{(-1)^d\Sigma^{-d}(\xi^{d+1})}\ar[d]^{(-1)^d\nu\circ\Sigma^{-d}(\gamma^{d+1})} &&& X^0\ar[r]^{\xi^0}\ar@{=}[d] & X^1\ar[r]\ar@{-->}[d]^{\alpha^1}& \cdots\ar[r]^{\xi^{d-1}} & X^d\ar[r]^{\xi^d}\ar@{-->}[d]^{\alpha^d} & W\ar[d]^{(-1)^d\Sigma^d(\nu)\circ\gamma^{d+1}} \\
W^0\ar[rrr]_{\varphi} &&& X^0\ar[r]_{\delta^0} & Y^1\ar[r] &\cdots\ar[r]_-{\delta^{d-1}} & Y^d\ar[r]_-{\delta^d} & \Sigma^d W^0.
}
\end{align}
%X^{d-1}\ar[r]^{\xi^{d-1}}\ar@{-->}[d]^{\alpha^{d-1}} &
%Y^{d-1}\ar[r]_{\omega^{d-1}} &
For any $W'$ in $\mathcal{W}$, consider the exact sequence
\begin{align*}
    \Hom (W', Y^{d-1})\xrightarrow{\widetilde{\delta^{d-1}}}\Hom(W',Y^d)\xrightarrow{\widetilde{\delta^d}}\Hom(W',\Sigma^d W^0)\xrightarrow{\widetilde{(-1)^d\Sigma^d(\varphi)}}\Hom(W',\Sigma^d X^0),
\end{align*}
where, for a morphism $\eta$ we use the notation $\widetilde{\eta}:=\Hom(W',\eta)$ for readability.
Note that $Y^1\rightarrow \dots \rightarrow Y^d$ is in $\mathcal{W}$-exact by \cite[lemma\ 2.1]{J}. Hence $\widetilde{\delta^{d-1}}$ is surjective, so $\widetilde{\delta^d}$ is the zero map and $\widetilde{(-1)^d\Sigma^d(\varphi)}$ is injective.

Let $\phi:W'\rightarrow W$ be a morphism in $\mathcal{W}$ which is not a split epimorphism. As $\xi^d$ is right almost split, there exists a morphism $\eta:W'\rightarrow X^d$ such that $\phi=\xi^d\circ\eta$. Consider $\delta^d\alpha^d\eta\in\Hom(W',\Sigma^d W^0)$ and note that
\begin{align*}
    \widetilde{(-1)^d\Sigma^d(\varphi)}(\delta^d\alpha^d\eta)=(-1)^d\Sigma^d(\varphi)\circ\delta^d \alpha^d\eta=0\circ\alpha^d\eta= 0,
\end{align*}
where $(-1)^d\Sigma^d(\varphi)\circ\delta^d=0$ by Lemma \ref{lemma_consecutive}. Then, by injectivity of $\widetilde{(-1)^d\Sigma^d(\varphi)}$, we conclude that $\delta^d \alpha^d\eta=0$.

By commutativity of (\ref{diagram_atob}), we have
\begin{align*}
    0= \delta^d \alpha^d\eta=(-1)^d\Sigma^d(\nu)\gamma^{d+1}\xi^d\eta=(-1)^d\Sigma^d(\nu)\gamma^{d+1}\phi.
\end{align*}
Then we obtain a commutative diagram:
\begin{align*}
\xymatrix {
&&&&&& W'\ar[d]^{(-1)^d\phi}\ar[rrd]^0 \ar@{-->}[lld]_{\exists \beta'} \\
W^0\ar[r]_{\omega^0} & W^1\ar[r]_{\omega^1}& W^2\ar[r] &\cdots\ar[r] & W^d\ar[rr]_{\omega^d} && W\ar[rr]_{\Sigma^d(\nu)\gamma^{d+1}} &&\Sigma^d W^0,
}
\end{align*}
where $\beta':W'\rightarrow W^d$ exists by Lemma \ref{lemma_zeroiffexists1}. Letting $\beta:=(-1)^d\beta'$, we have $\phi=\omega^d\circ\beta$. Hence $\omega^d$ is right almost split in $\mathcal{W}$ as we wished and (a) implies (b).

We now prove that (b) implies (a). Suppose that we have an Auslander-Reiten $(d+2)$-angle in $\mathcal{W}$ of the form
  \begin{align*}
\xymatrix {
\epsilon':& W^0\ar[r]^{\omega^0} & W^1\ar[r]^{\omega^1}& W^2\ar[r] &\cdots\ar[r] & W^d\ar[r]^{\omega^d} & W\ar[r]^-{\omega^{d+1}} &\Sigma^d W^0.
}
\end{align*}
Since $\omega^d$ is not a split epimorphism and $\xi^d$ is right almost split, there is a morphism $\varphi^d:W^d\rightarrow X^d$ such that $\xi^d\circ\varphi^d=\omega^d$. Then there are morphisms $\varphi^0,\dots,\,\varphi^{d-1}$ making the following diagram commutative:
\begin{align}\label{diagram_btoa}
    \xymatrix {
\Sigma^{-d}W\ar[rr]^-{(-1)^d\Sigma^{-d}\omega^{d+1}}\ar@{=}[d] && W^0\ar[r]^{\omega^0}\ar@{-->}[d]^{\varphi^0}& W^1\ar[r]\ar@{-->}[d]^{\varphi^1} &\cdots\ar[r] & W^{d-1}\ar[r]^{\omega^{d-1}}\ar@{-->}[d]^{\varphi^{d-1}} & W^d\ar[r]^{\omega^{d}}\ar[d]^{\varphi^{d}} &W\ar@{=}[d]\\
\Sigma^{-d}W\ar[rr]_{(-1)^d\Sigma^{-d}\xi^{d+1}} && X^0\ar[r]_{\xi^0}& X^1\ar[r] &\cdots\ar[r] & X^{d-1}\ar[r]_{\xi^{d-1}} & X^d\ar[r]_{\xi^d} &W.
}
\end{align}
We show that $\varphi^0:W^0\rightarrow X^0$ is a $\mathcal{W}$-cover. First note that commutativity of (\ref{diagram_btoa}) and the fact that $\xi^{d+1}$ is non-zero implies that $\varphi^0$ is non-zero. Moreover, by Lemma \ref{lemma_AReq_W}, we know that $\End(W^0)$ is local. Hence, by the dual of \cite[lemma\ 2.4]{KH}, it follows that $\varphi^0$ is right minimal. So it remains to show that $\varphi^0$ is a $\mathcal{W}$-precover.

Suppose that $U^0$ in $\mathcal{W}$ and a morphism $\gamma^0: U^0\rightarrow X^0$ are given. We want to prove that $\gamma^0$ factors through $\varphi^0$. The case $U^0=0$ is trivial, so suppose that $U^0$ is non-zero.

Take a linear map $\psi:\Hom(\Sigma^{-d}W,X^0)\rightarrow k$ with $\psi(\Sigma^{-d}(\xi^{d+1}))\neq 0$. Define a bilinear map
\begin{align*}
 q: \Hom(\Sigma^{-d}&W, U^0)\times \Hom(U^0,W^0)\rightarrow k,\\
    &q(\phi,\alpha)=\psi(\varphi^0\alpha\phi).
\end{align*}
We show that if $\phi\neq 0$, then there exists an $\alpha$ such that $q(\phi,\alpha)\neq 0$. Let $\phi\in\Hom(\Sigma^{-d}W,U^0)$ be non-zero and extend $\Sigma^d(\phi)$ to a $(d+2)$-angle of the form
\begin{align*}
\xymatrix {
U^0\ar[r]^{\nu^0} & U^1\ar[r]^{\nu^1}& U^2\ar[r] &\cdots\ar[r] & U^d\ar[r]^{\nu^d} & W\ar[r]^-{\Sigma^d\phi} &\Sigma^d U^0,
}
\end{align*}
where, since $W,\,U^0$ are in $\mathcal{W}$, we can choose $U^1,\dots,\,U^d$ in $\mathcal{W}$. Note that, as $\Sigma^d\phi$ is non-zero and by Lemma \ref{lemma_epimono}, then $\nu^d$ is not a split epimorphism. So there is $\eta^d:U^d\rightarrow W^d$ such that $\nu^d=\omega^d\circ\eta^d$. Hence $\omega^{d+1}\nu^d=\omega^{d+1}\omega^d\eta^d=0$ by Lemma \ref{lemma_consecutive}.
Then we have a commutative diagram:
\begin{align*}
\xymatrix {
U^0\ar[r]^{\nu^0} & U^1\ar[r]^{\nu^1}& U^2\ar[r] &\cdots\ar[r] & U^d\ar[rr]^{\nu^d}\ar[rrd]_0 && W\ar[rr]^{\Sigma^d(\phi)}\ar[d]^{\omega^{d+1}} &&\Sigma^d U^0\ar@{-->}[lld]^{\exists \Sigma^d(\eta^0)},\\
&&&&&& \Sigma^d W^0
}
\end{align*}
where $\Sigma^d(\eta^0)$ exists by Lemma \ref{lemma_zeroiffexists2}. Note that $\eta^0\circ \phi=\Sigma^{-d}(\omega^{d+1})$. Then the element $\eta^0$ in $\Hom(U^0,W^0)$ is such that 
\begin{align*}
    q(\phi, \eta^0)=\psi(\varphi^0\eta^0\phi)=\psi(\varphi^0\Sigma^{-d}(\omega^{d+1}))=\psi(\Sigma^{-d}\xi^{d+1})\neq 0,
\end{align*}
so we have established the desired property of $q$.

Consider the linear map
\begin{align*}
    \varphi: \Hom&(\Sigma^{-d}W,U^0)\rightarrow k, \\
    &\varphi(\phi)=\psi(\gamma^0\phi).
\end{align*}
By \cite[lemma\ 2.5]{JP}, there is an element $\alpha\in\Hom(U^0,W^0)$ such that $\varphi(-)=q(-,\alpha)$. Then, by the definitions of $\varphi$ and $q$, for any $\phi\in\Hom(\Sigma^{-1}W,U^0)$, we have
\begin{align}\label{eqn_btoa}
    \psi(\gamma^0\phi)=\psi(\varphi^0\alpha\phi).
\end{align}
Since $\epsilon$ is an Auslander-Reiten $(d+2)$-angle in $\mathcal{M}$, then so is $\overline{\epsilon}:=(-1)^d\Sigma^{-d}\epsilon$. By Lemma \ref{lemma_sublemma}, we conclude that the bilinear map
\begin{align*}
 p: \Hom(\Sigma^{-d}&W, U^0)\times \Hom(U^0,X^0)\rightarrow k,\\
    &p(\phi,\delta)=\psi(\delta\circ\phi)
\end{align*}
is non-degenerate. Hence (\ref{eqn_btoa}) implies that $\gamma^0=\varphi^0\circ \alpha$, that is $\gamma^0$ factors through $\varphi^0$ as we wished.
\end{proof}

\section{A class of examples}\label{section_example}
In this section, we present a class of examples using $\mathcal{F}$ as described in \cite[section\ 4]{V} and\cite[section\ 7]{HJV}. We first give a full description of the wide subcategories $\overline{\mathcal{W}}$ of the $(d+2)$-angulated category $\overline{\mathcal{F}}=\add \{ \Sigma^{id}\mathcal{F}\mid i\in \mathbb{Z} \}$, using Theorem A. We then apply Theorem B to find the Auslander-Reiten $(d+2)$-angles in these subcategories $\overline{\mathcal{W}}$.

Note that \ref{defn_exampleV}-\ref{rmk_d-seq} will summarise results due to \cite[section\ 4]{V} and \cite[section\ 7]{HJV}.

\begin{defn}\label{defn_exampleV}
Let $d\geq 2,\,l\geq 2,\, m\geq 3$ be integers such that $d$ is even and 
\begin{align*}
    \frac{m-1}{l}=\frac{d}{2}.
\end{align*}
Let $Q$ be the quiver
\begin{align*}
    m\rightarrow m-1\rightarrow\cdots\rightarrow 2\rightarrow 1,
\end{align*}
and set $\Phi:= k Q / (\rad_{k Q})^l$.

For vertex $i$, let $p_i$ denote the corresponding projective $\Phi$-module and $q_i$ denote the corresponding injective $\Phi$-module. Note that $p_i=q_{i-l+1}$ for $l\leq i\leq m$. Let
\begin{align*}
    f_i:=
    \begin{cases}
p_i &\text{ for } 1\leq i\leq m,\\
q_{i-l+1} &\text{ for } m+1\leq i\leq m+l-1, \\
0 &\text{ for } i\leq 0 \text{ and } i\geq m+l.
\end{cases}
\end{align*}
Then we have a $d$-cluster tilting module
\begin{align*}
    f=\bigoplus_{i=1}^{m+l-1} f_i,
\end{align*}
and a $d$-cluster tilting subcategory $\mathcal{F}=\add(f)\subseteq \mmod (\Phi)$.
Moreover, we have that
\begin{align*}
    \dim _k \mathcal{F}(f_i,f_j)=
    \begin{cases}
1 &\text{ if } 0\leq j-i\leq l-1,\\
0 &\text{ otherwise.}
\end{cases}
\end{align*}
\end{defn}
\begin{exmp}\label{exmp_vaso}
For $d=4$, $l=4$ and $m=9$, the Auslander-Reiten quiver of $\Phi$ is:
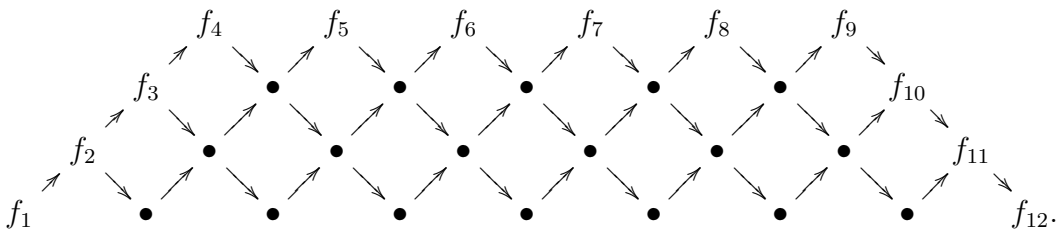
\begin{align*}
    \xymatrix  @!0{
    &&& f_4\ar[rd] && f_5\ar[rd] && f_6\ar[rd] && f_7\ar[rd] && f_8\ar[rd] && f_9\ar[rd] \\
    && f_3\ar[ru]\ar[rd] && \bullet\ar[ru]\ar[rd] &&  \bullet\ar[ru]\ar[rd] && \bullet\ar[ru]\ar[rd] && \bullet\ar[ru]\ar[rd] && \bullet\ar[ru]\ar[rd] && f_{10}\ar[rd]\\
    & f_2\ar[ru]\ar[rd] && \bullet\ar[ru]\ar[rd] && \bullet\ar[ru]\ar[rd] && \bullet\ar[ru]\ar[rd] && \bullet\ar[ru]\ar[rd] && \bullet\ar[ru]\ar[rd] && \bullet\ar[ru]\ar[rd]&& f_{11}\ar[rd]\\
    f_1\ar[ru] && \bullet\ar[ru] && \bullet\ar[ru] && \bullet\ar[ru] && \bullet\ar[ru] && \bullet\ar[ru] && \bullet\ar[ru] && \bullet\ar[ru] && f_{12}.
    }
\end{align*}
\end{exmp}
Consider now the $(d+2)$-angulated category $\overline{\mathcal{F}}=\add \{ \Sigma^{id}\mathcal{F}\mid i\in \mathbb{Z} \}\subseteq \mathcal{D}^b(\mmod \Phi)$. It has the quiver
\begin{align}\label{diagram_quiver}
    \cdots \Sigma^{-d} f_1\rightarrow\cdots\rightarrow \Sigma^{-d} f_{m+l-1}\rightarrow f_1\rightarrow\cdots\rightarrow f_{m+l-1}\rightarrow \Sigma^d f_1\rightarrow\cdots,
\end{align}
where the composition of $l$ consecutive maps is zero, see \cite[proposition\ A.11]{JL}.
\begin{remark}\label{rmk_d-seq}
For any non-zero morphism $\mu:f_i\rightarrow f_j$, where $i\neq j$, there is an exact sequence of the form:
\begin{align*}
    \cdots\rightarrow f_{j-2l}\rightarrow f_{i-l}\rightarrow f_{j-l}\rightarrow f_i\xrightarrow{\mu} f_j\rightarrow f_{i+l}\rightarrow f_{j+l}\rightarrow f_{i+2l}\rightarrow\cdots .
\end{align*}
Note that this sequence terminates both on the right and on the left giving a $d$-exact sequence containing $\mu$:
\begin{align*}
    0\rightarrow f_x\rightarrow \cdots\rightarrow f_{j-l} \rightarrow f_i\xrightarrow{\mu} f_j\rightarrow f_{i+l}\rightarrow \cdots \rightarrow f_y\rightarrow 0.
\end{align*}
Note that the cases when $f_{i+l}=0$ (or $f_{j-l}=0$) are allowed and correspond to $f_i$ being injective non-projective (or $f_j$ being projective non-injective, respectively). In these cases, $\mu$ is surjective (or injective, respectively).

By Remark \ref{rmk_angle_seq}, this gives a $(d+2)$-angle in $\overline{\mathcal{F}}$  of the form:
\begin{align*}
    f_x\rightarrow \cdots\rightarrow f_{j-l} \rightarrow f_i\xrightarrow{\mu} f_j\rightarrow f_{i+l}\rightarrow \cdots \rightarrow f_y\rightarrow \Sigma^d f_x.
\end{align*}
Note that we can rotate this $(d+2)$-angle to make $f_j$ the end-term on the right.
\end{remark}
\begin{lemma}\label{lemma_AR_vaso}
Let $1\leq j\leq m+l-1$. Then there is an Auslander-Reiten $(d+2)$-angle in $\overline{\mathcal{F}}$ ending at $f_j$ of the form:
\begin{align}
&\Sigma^{-d} f_l\rightarrow \Sigma^{-d} f_{l+1}\rightarrow \cdots\rightarrow \Sigma^{-d} f_m \rightarrow \Sigma^{-d} f_{m+l-1}\xrightarrow{\mu} f_1\rightarrow &&\text{ if } j=1;\tag{a}\\
    &\Sigma^{-d} f_{j+l-1}\rightarrow \Sigma^{-d} f_{j+l}\rightarrow\cdots\rightarrow f_{j-1}\xrightarrow{\mu}  f_j\rightarrow  &&\text{ if } 2\leq j\leq m; \tag{b}\\
    &f_{j-m}\rightarrow f_{j+1-m}\rightarrow \cdots\rightarrow f_{j-l}\rightarrow f_{j-1}\xrightarrow{\mu} f_j\rightarrow  &&\text{ if } j\geq m+1. \tag{c}
\end{align}
\end{lemma}
\begin{proof}
First note that in any case, the complex is a $(d+2)$-angle in $\overline{\mathcal{F}}$ by Remark \ref{rmk_d-seq}. In fact, we can extend and rotate $\mu$ in cases (b) and (c) and $f_1\rightarrow f_l$ in case (a).

Moreover, since any morphism between two non-isomorphic indecomposable objects is in $\rad_{\overline{\mathcal{F}}}$, all the morphism of the $(d+2)$-angle are in $\rad_{\overline{\mathcal{F}}}$.

In cases (b) and (c), we have $\mu: f_{j-1}\rightarrow f_j$. This is not a split epimorphism as the only morphism of the form $f_j\rightarrow f_{j-1}$ is the zero morphism. 
Let $\overline{f}\in\overline{\mathcal{F}}$ and $\alpha:\overline{f}\rightarrow f_j$ be a non-zero morphism which is not a split epimorphism. Without loss of generality, assume that $\overline{f}$ is indecomposable. Note that, since $\alpha$ is not an isomorphism, then $\overline{f}\neq f_j$. Hence $\overline{f}$ is an object to the left of $f_j$ in the quiver (\ref{diagram_quiver}) and $\alpha$ factors through $\mu$.
Similarly, in case (a), $\mu:\Sigma^{-d} f_{m+l-1}\rightarrow f_1$ is not a split epimorphism and any morphism in $\overline{\mathcal{F}}$ ending at $f_1$ that is not a split epimorphism factors through $\mu$.

Hence in any case, $\mu$ is right almost split and the $(d+2)$-angle is an Auslander-Reiten $(d+2)$-angle by Lemma \ref{lemma_lras_gen}.
\end{proof}
\begin{exmp}[Continuing Example \ref{exmp_vaso}]\label{exmp_vaso2}
Using Lemma \ref{lemma_AR_vaso}, the following are Auslander-Reiten $6$-angles in $\overline{\mathcal{F}}$:
\begin{align*}
    &\Sigma^{-4} f_{4}\rightarrow \Sigma^{-4} f_{5}\rightarrow \Sigma^{-4} f_8\rightarrow \Sigma^{-4} f_{9}\rightarrow \Sigma^{-4} f_{12}\xrightarrow{\mu} f_1\rightarrow f_4, &&\text{ where } j=1;\tag{a}\\
    &\Sigma^{-4} f_{8}\rightarrow \Sigma^{-4} f_{9}\rightarrow \Sigma^{-4} f_{12}\rightarrow  f_{1}\rightarrow f_4\xrightarrow{\mu} f_5\rightarrow f_8,  &&\text{ where } j=5; \tag{b}\\
    &f_{1}\rightarrow f_{2}\rightarrow f_5\rightarrow f_{6}\rightarrow f_9\xrightarrow{\mu} f_{10}\rightarrow \Sigma^{4} f_1,  &&\text{ where } j=10. \tag{c}\\
\end{align*}
\end{exmp}

\begin{lemma}\label{rmk_wide_sn}
Let $\mathcal{V}\subseteq \overline{\mathcal{F}}$ be a wide subcategory.
Then, $\mathcal{V}=\overline{\mathcal{W}}=\add \{ \Sigma^{id}\mathcal{W}\mid i\in \mathbb{Z} \}$ for some wide subcategory $\mathcal{W}$ of $\mathcal{F}$. Moreover,
\begin{enumerate}[label=(\alph*)]
    \item $\mathcal{W}$ is semisimple if and only if for all distinct $f_i,\,f_j$ in $\mathcal{W}$, we have $l\leq \mid i-j\mid\leq m-1$;
    \item $\mathcal{W}$ is non-semisimple if and only if it is \textit{$l$-periodic}, \textbf{i.e.} $0\neq f_q\in\mathcal{W}$ implies $f_{q+rl}\in\mathcal{W}$ for all $r\in\mathbb{Z}$.
\end{enumerate}
%By Theorem A, all wide subcategories of $\overline{\mathcal{F}}$ have the form $\overline{\mathcal{W}}=\add \{ \Sigma^{id}\mathcal{W}\mid i\in \mathbb{Z} \}$ for $\mathcal{W}$ as in case (1) or (2) above.
\end{lemma}

\begin{proof}
The fact that $\mathcal{V}=\overline{\mathcal{W}}$ follows from Theorem A. The rest of the lemma follows from \cite[section\ 7]{HJV}.
\end{proof}

\begin{lemma}
Let $\overline{\mathcal{W}}\subseteq\overline{\mathcal{F}}$ be a wide subcategory, where $\mathcal{W}\subseteq \mathcal{F}$ is semisimple. Let $f_j\in\mathcal{W}$ and suppose $\overline{f}$ is the initial object of the Auslander-Reiten $(d+2)$-angle in $\overline{\mathcal{F}}$ ending at $f_j$. Then
    \begin{align*}
        \Sigma^{-d} f_j\rightarrow 0\rightarrow \cdots\rightarrow 0\rightarrow f_j\xrightarrow{1_{f_{j}}} f_j
    \end{align*}
    is an Auslander-Reiten $(d+2)$-angle in $\overline{\mathcal{W}}$ and $\Sigma^{-d} f_j\rightarrow \overline{f}$ is a $\overline{\mathcal{W}}$-cover.
\end{lemma}

\begin{proof}
We claim that the only non-zero morphisms in $\overline{\mathcal{W}}$ are scalar multiples of the identity morphisms. If $f_i,\,f_k$ are two distinct objects in $\mathcal{W}$, then
\begin{align*}
    \overline{\mathcal{W}}(f_i,f_k)=0 \text{ since } l\leq \mid i-k\mid.
\end{align*}
Suppose for a contradiction that $\overline{\mathcal{W}}(f_i,\Sigma^d f_k)$ is non-zero. Then, since the composition of $l$ consecutive arrows in diagram (\ref{diagram_quiver}) is zero, we must have $k<i$ and there is a sequence of at most $l$ objects of the form:
\begin{align*}
    f_i\rightarrow \cdots \rightarrow f_{m+l-1}\rightarrow \Sigma^d f_1 \rightarrow \cdots \rightarrow \Sigma^d f_k.
\end{align*}
But then $m+l-i+k\leq l$, contradicting the fact that $ i-k \leq m-1$. Hence we proved our claim and so $0\rightarrow f_j$ is right almost split in $\overline{\mathcal{W}}$. Then
 \begin{align*}
        \Sigma^{-d} f_j\rightarrow 0\rightarrow \cdots\rightarrow 0\rightarrow f_j\xrightarrow{1_{f_{j}}} f_j
    \end{align*}
    is an Auslander-Reiten $(d+2)$-angle in $\overline{\mathcal{W}}$ and $\Sigma^{-d} f_j\rightarrow \overline{f}$ is a $\overline{\mathcal{W}}$-cover, by Theorem B.
\end{proof}

\begin{lemma}\label{lemma_cover_vaso}
Let $\overline{\mathcal{W}}\subseteq\overline{\mathcal{F}}$ be a wide subcategory, where $\mathcal{W}\subseteq \mathcal{F}$ is non-semisimple. Let $f_j\in\mathcal{W}$ and suppose $\overline{f}$ is the initial object of the Auslander-Reiten $(d+2)$-angle in $\overline{\mathcal{F}}$ ending at $f_j$.

Starting from $\overline{f}$ and moving left in the quiver (\ref{diagram_quiver}), let $w$ be the first object found which is in $\overline{\mathcal{W}}$. Then there is a $\overline{\mathcal{W}}$-cover $w\rightarrow \overline{f}$, and we have an Auslander-Reiten $(d+2)$-angle in $\overline{\mathcal{W}}$ of the form:
 \begin{align*}
            w \rightarrow \cdots \rightarrow f_j\rightarrow \Sigma^d w.
        \end{align*}
   
\end{lemma}
\begin{remark}
Note that $\overline{f}$ is described in cases (a), (b), (c) of Lemma \ref{lemma_AR_vaso} for all possible values of $j$.
In cases (b) and (c), so $\overline{f}=\Sigma^{-d} f_{j+l-1}$, we have $w=\Sigma^{-d}f_p$, for
\begin{align*}
p:=\text{max} \{ n\in\mathbb{Z}^{>0}\mid n\leq j+l-1,\, f_p\in\mathcal{W} \}.
\end{align*}
In case (a), so $\overline{f}=f_{j-m}$, then $w$ can be either of the form $f_p$ or $\Sigma^{-d} f_q$.

Moreover, once $w$ is found, Remark \ref{rmk_d-seq}  can be used to find the Auslander-Reiten $(d+2)$-angle in $\overline{\mathcal{W}}$. Note that the latter has half of its objects equal to the ones in the  Auslander-Reiten $(d+2)$-angle in $\overline{\mathcal{F}}$ ending at $f_j$, \textbf{i.e.} $f_j$ and every second of the terms to its left. The remaining objects are obtained by replacing $\overline{f}$ with $w$ and, at every step, shifting by $l$ objects in diagram (\ref{diagram_quiver}).
\end{remark}    
    
\begin{proof}[Proof of Lemma \ref{lemma_cover_vaso}]
Given any object $g$ in (\ref{diagram_quiver}), the indecomposable objects in $\overline{\mathcal{F}}$ having non-zero morphism into $g$ are exactly $g$ and the $l-1$ objects to its left in the quiver.

Consider $g=\overline{f}$  and note that, since $\mathcal{W}\neq 0$ is $l$-periodic, then at least one of these $l$ objects is in $\overline{\mathcal{W}}$. Hence $w$ can be chosen as described in Lemma \ref{lemma_cover_vaso} with $w\rightarrow \overline{f}$ non-zero. Moreover, $\delta:w\rightarrow \overline{f}$ is a $\overline{\mathcal{W}}$-cover since all other morphisms from $\overline{\mathcal{W}}$ to $\overline{f}$ factor through $\delta$. The last part of the lemma follows from Theorem B.
\end{proof}

\begin{exmp}[Continuing Example \ref{exmp_vaso}]
Let $\mathcal{W}=\add \{ f_1,f_2,f_5,f_6,f_9,f_{10} \}$. Note this is $4$-periodic and hence $\overline{\mathcal{W}}\subseteq \overline{\mathcal{F}}$ is wide. 
Consider the $6$-angle (a) from Example \ref{exmp_vaso2}, where $f_1\in\overline{\mathcal{W}}$. Here, $\overline{f}=\Sigma^{-4} f_4$ has $\overline{\mathcal{W}}$-cover $w=\Sigma^{-4} f_2 \rightarrow\Sigma^{-4} f_4$. Then, we obtain the Auslander-Reiten $6$-angle in $\overline{\mathcal{W}}$:
\begin{align*}
    \Sigma^{-4} f_{2}\rightarrow \Sigma^{-4} f_{5}\rightarrow \Sigma^{-4} f_{6}\rightarrow \Sigma^{-4}  f_{9}\rightarrow \Sigma^{-4 }f_{10}\rightarrow f_1\rightarrow f_2.
\end{align*}

Similarly, starting from the $6$-angle (b) from Example \ref{exmp_vaso2}, we obtain the Auslander-Reiten $6$-angle in $\overline{\mathcal{W}}$:
\begin{align*}
\Sigma^{-4} f_{6}\rightarrow \Sigma^{-4} f_{9}\rightarrow \Sigma^{-4} f_{10}\rightarrow  f_{1}\rightarrow f_2\rightarrow f_5\rightarrow f_6.
\end{align*}
Finally, note that since all the objects in the  $6$-angle (c) in Example \ref{exmp_vaso2}  are in $\overline{\mathcal{W}}$, then 
\begin{align*}
    f_{1}\rightarrow f_{2}\rightarrow f_5\rightarrow f_{6}\rightarrow f_9\xrightarrow{\mu} f_{10}\rightarrow \Sigma^{4} f_1
\end{align*}
is also an Auslander-Reiten $6$-angle in $\overline{\mathcal{W}}$.
\end{exmp}

\end{document}